\documentclass[12pt]{amsart}

\usepackage{pb-diagram}

\usepackage{amsfonts}
\usepackage{amsmath}
\usepackage{amssymb}

\newcommand{\dbar}{\overline{\partial}}

\newcommand{\ddt}[1]{\frac{\partial #1}{\partial t}}
\newcommand{\R}{\mathcal{R}}
\newcommand{\Sc}{\mathcal{S}}

\newcommand{\A}{\mathcal{A}}

\newcommand{\ddbar}{\sqrt{-1}\partial\dbar}

\newtheorem{claim}{Claim}[section]
\newtheorem{theorem}{Theorem}[section]
\newtheorem{proposition}{Proposition}[section]
\newtheorem{lemma}{Lemma}[section]
\newtheorem{conjecture}{Conjecture}[section]
\newtheorem{example}{Example}[section]
\newtheorem{definition}{Definition}[section]
\newtheorem{corollary}{Corollary}[section]

\begin{document}

\title{Riemannian  geometry of K\"ahler-Einstein currents }

\author{Jian Song}

\address{Department of Mathematics, Rutgers University, Piscataway, NJ 08854}

\email{jiansong@math.rutgers.edu}

\thanks{Research supported in
part by National Science Foundation grants DMS-0847524.}

\begin{abstract} We study   Riemannian geometry of canonical K\"ahler-Einstein currents on projective Calabi-Yau varieties and canonical models of general type with crepant singularities. We prove that the metric completion of the regular part of such a canonical current is a compact metric length space homeomorphic to the original projective variety, with well-defined tangent cones. We also prove a special degeneration for K\"ahler-Einstein manifolds of general type as an approach to establish the  compactification of the moduli space of K\"ahler-Einstein manifolds of general type. A number of applications are given for degeneration of Calabi-Yau manifolds and the K\"ahler-Ricci flow on smooth minimal models of general type. 
\end{abstract}

\maketitle

{\footnotesize  \tableofcontents}


\section{Introduction}

Recent progress in the study of canonical metrics in K\"ahler geometry has revealed deep connections and interplay among nonlinear PDEs, Riemannian geometry and complex algebraic geometry. The Yau-Tian-Donaldson conjecture \cite{Y3, T3, D} predicts the relation between the existence of K\"ahler-Einstein metrics and the $K$-stability for Fano manifolds. The analytic minimal model program with Ricci flow proposed by the author and  Tian \cite{ST3} connects finite time singularity  of the K\"ahler-Ricci flow to geometric and birational surgeries,  and its  long time behavior to the existence of singular K\"ahler-Einstein metrics and the abundance conjecture. In particular, it is proposed by the author \cite{S3} that the K\"ahler-Ricci flow should give a global uniformization in terms of K\"ahler-Einstein metrics for projective varieties as well as a local uniformization in terms of the transition of shrinking and expanding solitons for singularities arising simultaneously from the K\"ahler-Ricci flow and birational transformation \cite{SW1, SW2, SW3, SoYu, S3}. 

The theory of Cheeger-Colding plays an important role in the recent fundamental work of Donaldson and Sun \cite{DS} to prove the partial $C^0$-estimate proposed by Tian \cite{T1}, for a family of polarized K\"ahler manifolds with uniform bounds for the volume,  diameter and Ricci curvature.  More precisely, let $\mathcal{K}(n, V, H, D)$ be the set of  $n$-dimensional polarized K\"ahler manifolds $(X, g)$ with the K\"ahler metric $g \in H^{1,1}(X, \mathbb{R})\cap H^2(X, \mathbb{Z})$ satisfying 
\begin{itemize}

\item $\int_{X} dV_g \leq V$, 

\smallskip

\item $diam_{g} (X) \leq D$,
\smallskip

\item $ -H \leq Ric(g)\leq H.$

\end{itemize} Then all $(X, g) \in \mathcal{K}(n, V, H, D)$ can be embedded  simultaneously in a fixed large projective space. Furthermore, if $(X_j, g_j)\in \mathcal{K}(n, V, H, D)$ are K\"ahler-Einstein manifolds, then after passing to a subsequence, $(X_j, g_j)$ converges in Gromov-Hausdorff topology to a compact metric length space homeomorphic to a projective variety with log terminal singularities, coupled with a canonical  singular K\"ahler-Einstein metric \cite{DS}. This compactness result suggests an analytic and Riemannian geometric approach to construct moduli spaces of K\"ahler-Einstein manifolds and their compactifications. There are many important applications. For example, one can apply such a partial $C^0$-estimate to prove a conjecture Candalas and de la Ossa \cite{CaO} for geometric transitions of Calabi-Yau manifolds by combining the results of Rong-Zhang \cite{RZ}. 

On the other hand, this a powerful theorem only applies to polarized manifolds with K\"ahler metrics in $H^2(X, \mathbb{Z})$. In addition, the diameter bound,  which is equivalent to a nonlocal collapsing condition at each point, is usually difficult to verify unless assuming a nonnegative condition on the Ricci curvature. 

Often canonical K\"ahler-Einstein metrics arise from degeneration of K\"ahler classes in $H^2(X, \mathbb{Q})$ or more generally in $H^2(X, \mathbb{R})$. For example, it is shown in \cite{EGZ} that there exists a unique Ricci-flat K\"ahler current in any polarization of a normal projective variety with log terminal singularities and numerically trivial canonical line bundle. Also  for any minimal model of general type, there exists a unique K\"ahler-Einstein current in the canonical class on its unique canonical model \cite{Ts, EGZ, Z}. Such K\"ahler-Einstein currents are unique and thus canonical, and they have bounded continuous local potentials and their associated K\"ahler-Einstein metric is smooth on the regular part of the underlying variety \cite{EGZ}. Hence the natural question is  to ask how such singular currents give global and local information in terms of both Riemannian and algebraic geometry. One would expect, for example, that the metric completion of these K\"ahler-Einstein metrics on the regular part of the underlying variety would coincide with the original algebraic varieties topologically and algebraically, and the tangent cone at each point is unique and corresponds to the local algebraic affine cone.  

One approach to understand   Riemannian geometry of such K\"ahler-Einstein currents on singular varieties is  to study a family of nonsingular curvature equations on the nonsingular models after  resolution of singularities and metric perturbation. Such a deformation involves with a family of K\"ahler metrics in $H^2(X, \mathbb{Q})$ or even $H^2(X, \mathbb{R})$. Thus we would hope to generalize the work of Donaldson-Sun \cite{DS} to $\mathbb{Q}$-polarizations with certain additional assumptions. We would also like remove the assumption on the uniform diameter or equivalently a uniform non-collapsing condition in \cite{DS} in certain situation. 

In this paper, we make an attempt to apply techniques from pluripotential theory, nonlinear PDEs, Hormander's $L^2$ theory and the Cheeger-Colding theory to understand both local and global geometry of K\"ahler-Einstein currents on projective varieties with crepant singularities as well as certain special degeneration of canonical K\"ahler-Einstein manifolds.  Our goal is to establish the equivalence between the analytic weak solutions of K\"ahler-Einstein equations (or degenerate complex Monge-Ampere equations) and the weak metric limits of Cheeger-Colding theory,  and to show that such weak solutions must be strong solutions whose geometric structure coincides with their algebraic structure.

The following is our first main result. 

\begin{theorem}\label{main1} Let $X$ be an $n$-dimensional projective Calabi-Yau variety with crepant singularities and $L\rightarrow X$ an ample $\mathbb{Q}$-line bundle. Then there exists a unique Ricci-flat current $\omega_{KE} \in c_1(L)$ on $X$ with bounded local potentials satisfying

\begin{enumerate}

\item $\omega_{KE}$ is smooth on $X_{reg}$  and on $X_{reg}$,  
$$
Ric(g_{KE}) = 0, 
$$
where $g_{KE}$ is the K\"ahler metric  associated to $\omega_{KE}$ and $X_{reg}$ is  the regular part of the projective variety $X$.

\medskip

\item $(\hat X, d)$, the metric completion of $(X_{reg}, g_{KE})$, is a compact metrc length space  homeomorphic to $X$ as a projective variety. 

\medskip

\item Let $\hat X = \mathcal{R} \cup \mathcal{S}$ with $\mathcal{R}$ as the regular set of $(\hat X, d)$ and $\mathcal{S}$ as the singular set. Then $\mathcal{R}$ is an open dense convex set of $(\hat X,d)$ and $\mathcal{S}$ is a closed set of Hausdorff dimension no greater than $2n-4$.

\medskip

\item $\mathcal{R}= X_{reg}$.

\end{enumerate}
\end{theorem}

The regular set  $\mathcal{R}$ is the set of all points in $(\hat X,d)$ whose tangent cones are Euclidean $\mathbb{R}^{2n}$. We will show that all holomorphic sections of $L^k$ on $X_{reg}$ continuously  extend to the metric space $(\hat X, d)$ globally and the linear system $| L^k |$ induces a Lipschitz map 
$$\Phi: \hat X \rightarrow \mathbb{CP}^{d_k}$$ for some sufficiently large $k$, where $d_k +1= \dim H^0(X, L^k) $. In particular, the image of $\Phi$ is the projective variety $X$ itself and $\Phi$ is a homeomorphism (cf. section 3.3  and 3.5). Here the topology of $X$ is   induced from its project  embedding in some projective space with its Fubini-Study metric.   We remark that the existence of $\omega_{KE}$ in Theorem \ref{main1} is due to \cite{EGZ, Z} generalizing Kolodziej's fundamental work \cite{Kol1} to the degenerate case. In this special setting of crepant singularities, the same existence result  is also derived in \cite{To1, ZY} applying the family version of Moser's iteration due to Yau \cite{Y1}. Also it is shown by Rong-Zhang \cite{RZ} that $ \mathcal{R} = X_{reg}$. Our main contribution is to identify the compact metric space $(\hat X, d)$ with the original projective variety $X$ topologically via the global morphism $\Phi$. 

Immediately, we can apply Theorem \ref{main1} to understand certain degeneration of Calabi-Yau manifolds studied in \cite{To1, ZY} and show that such an  algebraic degeneration coincides with the analytic/geometric degeneration. 

\begin{corollary} \label{main3} Let $X$ be an $n$-dimensional  projective Calabi-Yau manifold. Suppose $g_t $ is a family of Ricci-flat K\"ahler metrics for $t\in (0, 1]$ and $[g_t]$ converges to a class $\alpha \in H^{1,1}(X, \mathbb{Q})$ associated to a big Cartier $\mathbb{Q}$-divisor. Then $(X, g_t)$ converges in Gromov-Hausdorff topology to a unique projective Calabi-Yau variety $(Y, d_Y)$ satisfying the conclusions of Theorem \ref{main1}. 

\end{corollary}

It is rather easy to guess the limit of such a degeneration, which is the projective Calabi-Yau variety $Y$ induced from the linear system of $|k \alpha|$ for some sufficiently large $k$. It is already shown in \cite{RZ} that $(X, g_t)$ converges to a unique compact metric space as the metric completion of the smooth K\"ahler-Einstein metric on the regular part of $Y$. Our contribution is to identify algebraic limit $Y$ with the Riemannian geometric limit of $(X, g_t)$.

We can also apply Theorem \ref{main1} and Corollary \ref{main3} to prove a conjecture of Candalas and de la Ossa \cite{CaO} for smooth flops strengthening earlier results of \cite{RuZ, RZ}. 

\begin{corollary} \label{main4}  Suppose $X$ and $X'$ are two $n$-dimensional smooth projective Calabi-Yau manifolds related by a flop 
\begin{equation}
\begin{diagram}\label{conflop}
\node{X} \arrow{se,b,}{f} \arrow[2]{e,t,..}{ }     \node[2]{X'} \arrow{sw,r}{f'} \\
\node[2]{Y}
\end{diagram}
\end{equation}

Then there exist a family of Ricci-flat metric spaces $(X_t, g_t)$ for $t\in [-1, 1]$ satisfying 

\begin{enumerate}

\item $X_t =X$ for $t\in [-1,0)$ and $X_t= X'$ for $(0, 1]$. $g_t$ is a smooth family of Ricci-flat K\"ahler metrics on $X_t$ for $t\in [-1, 1]\setminus \{0\}$,
\medskip

\item $(X_0, g_0)$ is a projective Calabi-Yau variety satisfying the conclusions of Theorem \ref{main1} with $X_0$ homeomorphic to $Y$.

\item $(X_t, g_t)$ is a continuous path in Gromov-Hausdorff topology. In particular, $g_t$ is smooth in $t$ except at the singular set of $Y$. 

\end{enumerate}

\end{corollary}

We remark that  the results of \cite{RZ}, combined with \cite{DS}, settle Candelas and de la Ossa's conjecture for geometric transitions of smooth projective Calabi-Yau manifolds. Since not all Calabi-Yau varieties with crepant singularities admit a projective smoothing \cite{Fr, T2}, the compactness result of \cite{DS} cannot be applied for smooth flops between projective Calabi-Yau manifolds.  A special case of Corollary \ref{main3} and Corollary \ref{main4} is proved by the author for Calabi-Yau conifolds \cite{S2}.  

We would like to obtain an analogue of Theorem \ref{main1}  for canonical models of general type with crepant singularities. A canonical model of general type is a projective variety $X$ with log terminal singularities such that the canonical divisor $K_X$ is ample. From the finite generation of canonical rings \cite{BCHM, Siu}, for any projective variety $X$ of general type,  there exists a minimal model of $X_{min}$ birationally equivalent to $X$, as well as a unique canonical model $X_{can}$.  It is now well-known from \cite{EGZ} that there exists a unique canonical K\"ahler-Einstein current with bounded local potentials in the canonical class on a canonical model of general type with log terminal singularities. In particular, the canonical model $X_{can}$ of a smooth minimal model $X_{min}$ of general type must admit crepant singularities as the pluricanonical map
$$\pi: X_{min} \rightarrow X_{can}$$ is a crepant resolution of $X_{can}$. Conversely, any canonical model of general type with crepant singularities can be derived by the contraction of a  smooth minimal model. The following theorem aimes to understand the global geometry of the K\"ahler-Einstein current on canonical models of general type with crepant singularities. 

\begin{theorem}\label{main2} Let $X$ be an $n$-dimensional normal projective variety with crepant singularities and ample canonical divisor $K_X$.    
Then there exists a unique K\"ahler-Einstein current $\omega_{KE} \in -c_1(X)$ on $X$ with bounded local potentials satisfying

\begin{enumerate}

\item $\omega_{KE} \in C^\infty(X_{reg})$ and its associated K\"ahler metric $g_{KE}$ on $X_{reg}$ satisfies
$$Ric(g_{KE} )= - g_{KE}. $$ 

\item $(\hat X, d)$, the metric completion of $(X_{reg}, g)$, is a compact metric length space  homeomorphic to $X$ as a projective variety. 

\medskip

\item Let $\hat X = \mathcal{R} \cup \mathcal{S}$ with $\mathcal{R}$ as the regular set of $(\hat X, d)$ and $\mathcal{S}$ as the singular set. Then $\mathcal{R}$ is an open dense convex set of $(\hat X,d)$ and $\mathcal{S}$ is a closed set of Hausdorff dimension no greater than $2n-4$.

\medskip

\item $ \mathcal{R} = X_{reg}$.

\end{enumerate}

\end{theorem}

The existence of $\omega_{KE}$ is due to \cite{Ts, EGZ, Z}. The proof of Theorem \ref{main2} follows the spirit of Theorem \ref{main1}. However, the main difficulty is that a priori, one does not have a uniform diameter bound, unlike the case of Theorem \ref{main1}, where a uniform diameter bound can be  achieved  from Yau's volume comparison \cite{SY}. One of our main contribution is to establish a uniform diameter bound,  for the limiting metric space constructed by Tian and Wang \cite{TW}. To achieve this, we have to apply the proof of Theorem \ref{main1} in a local setting and apply an analogue of quantitative metric estimate developed by the author and Weinkove in \cite{SW1, SW2, SW3} through algebraic blow-ups at the singular set.  

We can immediately apply Theorem \ref{main2} to the K\"ahler-Ricci flow on smooth minimal models of general type. 

\begin{corollary} \label{main5}

Let $X$ be a smooth minimal model of general type and let $\pi: X \rightarrow X_{can}$ be the unique birational morphism from $X$ to its canonical model $X_{can}$ induced by the pluricanonical system of $ X$. 
Then the normalized K\"ahler-Ricci flow 
$$\ddt{g} = -Ric(g) - g, ~ g(0) = g_0$$ for any smooth initial K\"ahler metric $g_0$,  admits a smooth solution $g(t)$ for $t\in [0, \infty)$ and $g(t)$ converges smoothly to a K\"ahler metric $g_\infty$ on $X\setminus E$, where $E$ is the exceptional locus of $\pi$. In particular,  $(\hat X, d)$, the metric completion of $(X\setminus E, g_\infty)$ is a singular K\"ahler-Einstein metric length space homeomorphic to the canonical $X_{can}$.

\end{corollary}

We remark that the long time existence and local smooth convergence is due to Tsuji \cite{Ts}. Also $(\hat X, d)$ satisfies all the conclusions in Theorem \ref{main4}. 

There are various possible generalizations of Theorem \ref{main1} and Theorem \ref{main4}. For example, one can consider the twisted K\"ahler-Einstein equation 
$$Ric(\omega) = \lambda \omega + \alpha $$
on a projective variety $X$, where $\alpha$ is nonnegative smooth closed $(1,1)$-form or or a current induced by a collection of effective prime divisors with coefficient in $(0,1)$ and simple normal crossings.  More interestingly, Theorem \ref{main2} should be generalized to log canonical pairs as discussed in section 4.6.  Another important direction is to study the Riemannian geometry of the metric spaces arising from a collapsing  family of Calabi-Yau manifolds studied in \cite{ST1, To2} or the canonical models of non-general type minimal models \cite{ST1, ST2}. The gradient estimates derived in section 3.2 and section 4.1 can be generalized in a very flexible way, especially for those complex Monge-Ampere equations whose corresponding curvature equation satisfies a Ricci lower bound. We expect such gradient estimates to hold for those twisted K\"ahler-Einstein metrics on canonical models of non-general type and the projective collapsing limit of Calabi-Yau manifolds. Theorem \ref{main2} can have other applications such as understanding the equality case of the Chern number inequality for general type varieties with mild singularities as suggested in \cite{TW}. 

Now we consider a different degeneration of K\"ahler-Einstein metrics in relation to the compactification of the moduli space of semi-log canonical models in algebraic geometry. We consider a flat projective family of canonical models
$$\pi: \mathcal{X} \rightarrow B,$$
where $B$ is an open disc in $\mathbb{C}$ and $X_t = \pi^{-1}(t)$ is an $n$-dimensional smooth projective manifold with $c_1(X) <0$ for $t\in B^*$. We assume that $\mathcal{X}$ has at worst canonical singularities, the central fibre $X_0$ has only one component with multiplicity one, and $X_0$ is a normal canonical model with   log canonical singularities. Furthermore, we assume that the relative canonical sheaf $K_{\mathcal{X}/B}$ is ample with 
$$K_{\mathcal{X}/B}|_{X_t} = K_{X_t}$$
for all $t\in B$.  We write $S_{lc}$ to be the singular set of points whose discrepancy is $-1$ (c.f. Definition \ref{slc}).

Then we have the following theorem. 

\begin{theorem} \label{main6} Let $\mathcal{X} \rightarrow B$ be the above flat family of $n$-dimensional polarized canonical models. Suppose that for any $k\in \mathbb{Z}^+$, $R>0$ and any smooth holomorphic section $\eta\in H^0(\mathcal{X}, (K_{\mathcal{X}/B})^k)$, 
$$\sup_{t\in B^*} \int_{B_{g_t}(p_t, R)} |\eta |_{X_t} |^{2/k} < \infty, $$
where $p_t\in X_t$ is a continuous family of points in $\mathcal{X}$ with $p_0\in (X_0)_{reg}$, $g_t \in -c_1(X_t)$ is the unique K\"ahler-Einstein metric on $X_t$ and $B_{g_t}(p_t, R)$ is the geodesic ball in $X_t$ of radius $R$ centered at $p_t$. Then the following hold.

\begin{enumerate}

\item There exists a unique K\"ahler-Einstein current $\omega_0\in c_1(X_0)$ satisfying
\begin{enumerate}

\smallskip

\item $\omega_0$ is smooth on $(X_0)_{reg}$, 

\smallskip
\item the local potential of $\omega_0$ is bounded on $X_0\setminus S_{lc}$ and tends to $-\infty$ along $S_{lc}$ if $S_{lc}\neq \phi$, 

\smallskip

\item $\int_{(X_0)_{reg}} \omega_0^n = (-c_1(X_0))^n$. 

\end{enumerate}

\medskip

\item The metric completion $(\hat X, d)$ of $((X_0)_{reg}, \omega_0) $ is either a compact metric length space homeomorphic to $X_0$ if $S_{lc}=\phi$ or a metric length space of infinite diameter homeomorphic to $X_0 \setminus S_{lc}$ if $S_{lc}\neq \phi$.  

\medskip

\item $\hat X = \mathcal{R}\cup \mathcal{S}$ with $\mathcal{R}$ being the regular set and $\mathcal{S}$ the singular set. Then $\mathcal{R} = (X_0)_{reg}$ is an open dense convex set and $\mathcal{S}$ is a closed set of Hausdorff dimension no greater than $2n-4$.

\medskip

\item $(X_t, g_t)$ converges to $(\hat X, d)$ in Gromov-Hausdorff topology  as $t\rightarrow 0$, where $g_t \in c_1(X_t)$ is the unique K\"ahler-Einstein metric on $X_t$ for $t\in B^*$. In particular, the convergence is smooth in $\mathcal{R}$.

\end{enumerate}

In particular, if for any $k$ and $\eta \in H^0(\mathcal{X}, (K_{\mathcal{X}/B})^k)$, $\sup_{t\in B^*} \int_{  X_t} |\eta |_{X_t} |^{2/k} <\infty$, then $X_0$ is log terminal and $S_{lc}=\phi$. 

\end{theorem}

The additional assumption on the integrability of holomorphic sections $\eta\in H^0(\mathcal{X}, (K_{\mathcal{X}/B})^k)$ should be removed. In particular, when $X_0$ is log terminal, the argument of Gross in \cite{RZ} (Theorem B.1 (ii)) might be applied. This assumption might also be replaced by a more reasonable and pure algebraic assumption for holomorphic sections in $H^0(X_0, (K_{X_0})^k)$ vanishing along $S_{lc}$ and their extension to $\mathcal{X}$. This will be studied in the sequel work. 

 Recent progress in algebraic geometry has established the compactification of semi-log canonical models \cite{K1}. Theorem \ref{main6} gives evidence that the moduli space of canonically polarized varieties is equivalent to the moduli space of Kaher-Einstein manifolds, as well as their algebraic and geometric compactifications. The proof of Theorem \ref{main6} can be adapted to more general central fibres with semi-log canonical singularities which correspond to a 'complete end' of the singular K\"ahler-Einstein metrics along the divisor. We also hope in the future study to obtain a compactness theorem for all smooth canonical models with uniform volume upper bound and identify any (pointed) Gromov-Hausdorff limit with a (quasi) projective variety. 

Finally, we would remark that in the case of global canonical K\"ahler-Einstein   currents studied in the paper, we have essentially established the following principle 
$$local ~potential \in L_{loc}^\infty ~\Longleftrightarrow ~ finite~distance. $$
The above observation is certainly not true in general, for example, on the pseudo-convex domains. 

\section{Preliminaries}

In this section, we will recall some basic definitions and notations . 

\begin{definition} Let $X$ be a normal projective variety $X$. If $K_X$ is a Cartier $\mathbb{Q}$-divisor and if $\pi: \tilde X \rightarrow X$ is a resolution of $X$, then there exist $a_i \in \mathbb{Q}$ with 
$$K_{\tilde X} = \pi^* K_X + \sum_i a_i E_i, $$
where $E_i$ ranges over all exceptional prime divisors of $\pi$. 
$X$ is said to have terminal singularities (canonical singularities,  log terminal singularities, log canonical), if all $a_i >0$ ( $a_i \geq 0$, $a_i > -1 $, $a_i\geq -1$).  In particular, $X$ is said to have crepant singularities if $$K_{\tilde X} = \pi^* K_X, $$ and $\pi$ is called a crepant resolution of $X$.

\end{definition}

\begin{example} All surface $A$-$D$-$E$ singularities are crepant singularities and there exists a unique crepant resolution for such singularities. 

\end{example}

\begin{example} Let $E= \mathcal{O}_{\mathbb{CP}^1}(-1)\oplus  \mathcal{O}_{\mathbb{CP}^1}(-1)$. The zero section of $E$ is a holomorphic $S^2$ with negative normal bundle. Then there exists a morphism $\pi: E \rightarrow \hat E$ by contracting the zero section of $E$, where $\hat E$ is an affine cone over $\mathbb{CP}^1\times \mathbb{CP}^1$.  The isolated singularity from the contraction of such a holomorphic $S^2$ is a crepant singularity and $\pi$ is a crepant resolution of $\hat E$.

\end{example}

Crepant singularities must have at worst canonical singularities by definition and conifold singularities are crepant and terminal singularities.

\begin{definition} Let $L\rightarrow X$ be a holomorphic line bundle over a normal projective variety $X$. $L$ is said to be semi-ample if the linear system $|L^k|$ is base point free for some $k\in \mathbb{Z}^+$. $L$ is said to be big if the Iitaka dimension of $L$ is equal to the dimension of $X$.

\end{definition}

We can define the semi-group 
\begin{equation}
\mathcal{F}(X, L) = \{ k \in \mathbb{Z}^+~|~ L^k~ is ~base~point~free \}. 
\end{equation}
For any $k\in \mathcal{F}(X,L)$, the linear system $|L^k|$ induces a morphism
$$\Phi_k= \Phi_{|L^k |}  :  X \rightarrow X_k \subset \mathbb{CP}^{d_k}, $$
where $d_k+1 = \dim H^0(X, L^k)$ and $X_k = \Phi_k(X)$. It is well-known \cite{La} that for sufficiently large $k, l \in \mathcal{F}(X, L)$, 
$$\pi_*\mathcal{O}_X = \mathcal{O}_{X_k} $$
and
$$\Phi_k = \Phi_l, ~~X_k = X_l .$$

Now we can define projective Calabi-Yau varieties.

\begin{definition} A projective Calabi-Yau variety $X$ is a normal projective variety with canonical singularities and numerically trivial canonical divisor $K_X$. 

\end{definition}

The following Theorem is a consequence of  Kawamata's base point free theorem  \cite{KMM}. 

\begin{proposition} \label{semicr} Let $L\rightarrow X$ be a holomorphic line bundle over a projective Calabi-Yau variety $X$. If $L$ is big and nef, it must be  semi-ample.

\end{proposition}

\begin{corollary} \label{crepsin} Suppose $L$ is a big and nef line bundle over a projective Calabi-Yau manifold $X$. Then the linear system $|L^k|$ induces a unique surjective birational morphism 
$$\Phi: X \rightarrow Y$$
for sufficiently large $k\in \mathcal{F}(X,L)$ such that $Y$ is a projective Calabi-Yau variety with crepant singularities.

\end{corollary}

\begin{example}  Let $Y$ be the hypersurface in $\mathbb{CP}^4$ defined by $$z_3 g(z_0, ..., z_4)+ z_4 h(z_0, ..., z_4)=0$$ with generic homogeneous polynomials $g, h$ of degree $4$ in $[z_0, z_1, ..., z_4] \in \mathbb{CP}^4$. The singular locus of $Y$ is given by $\{ z_3=z_4=g(z)=h(z)=0\} $, which consists of $16$ ordinary double points. The small resolution of the singularities of $Y$ gives rise to a smooth Calabi-Yau threefold $X$ and  $Y$ can also be smoothed to generic smooth quintic threefolds in $\mathbb{CP}^4$.

\end{example}

All projective Calabi-Yau varieties with conifiold singularities  admit a crepant resolution, although crepant resolutions are not necessarily unique and they are related by flops. Yau's celebrated solution \cite{Y1} to the Calabi conjecture \cite{C2} says that in any K\"ahler class of a Calabi-Yau manifold, there exists a unique smooth Ricci-flat K\"ahler metric.  The following theorem due to \cite{EGZ} is a generalization of Yau's theorem to projective Calabi-Yau varieties. 

\begin{theorem} Let $L\rightarrow X$ be an ample line bundle over a projective Calabi-Yau variety $X$ of $\dim  X=n$. Then there exists a unique Ricci-flat K\"ahler current $\omega \in c_1(L)$ satisfying

\begin{enumerate}

\item $\omega$ has bounded local potentials, i.e., for any point $p\in X$, there exists an open neighborhood $U$ of $p$ such that  $\omega= \ddbar \varphi$ for some $\varphi \in PSH(U)\cap L^\infty(U)$, 

\medskip

\item $\omega$ is smooth on $X_{reg}$, the regular part of $X$, 

\medskip

\item $Ric(\omega) =0$ on $X$. 

\medskip

\end{enumerate}

\end{theorem}

The Kodaira dimension of projective Calabi-Yau varieties is $0$. We now define minimal and canonical models of general type whose Kodaira dimension is equal to the complex dimension. We also remark that all line bundles considered in the paper are $\mathbb{Q}$-line bundles.

\begin{definition} A minimal model of general type is a normal projective variety whose canonical line bundle is big and nef. A canonical model of general type is a normal projective variety whose canonical line  bundle is ample. 
\end{definition}

If $X$ is a smooth minimal model of general type, then Kawamata's base point free theorem implies that $K_X$ is semi-ample. Therefore the pluricanonical system induces a birational morphism $$\pi: X \rightarrow X_{can}$$ from $X$ to its unique canonical model $X_{can}$. In particular, $X_{can}$ has crepant singularities. The following existence of a canonical K\"ahler-Einstein current on canonical models is due to \cite{EGZ, Z}. 

\begin{theorem} Let $X$ be a canonical model of general type with log terminal singularities. Then there exists a unique K\"ahler current $\omega_{KE} \in c_1(X)$ with bounded local potential  and smooth on $X_{reg}$, the regular part, such that 
$$Ric(\omega_{KE}) = - \ddbar \log (\omega_{KE}^n) = - \omega_{KE} $$
in distribution.

\end{theorem}

In fact, the current $\omega_{KE}$ can also be considered as a K\"ahler-Einstein current on $X_{can}$, the canonical model of $X$, after push-forward. 


\section{Projective Calabi-Yau varieties with crepant singularities}


In this section, we will prove Theorem \ref{main1}. Throughout the section, we assume that $X$ is an $n$-dimensional  normal projective Calabi-Yau variety with crepant singularities and $L \rightarrow X$ is an ample $\mathbb{Q}$-line bundle over $X$.  We let $$\pi: X' \rightarrow X $$ be a crepant resolution of $X$ and $L' = \pi^* L$. We also denote the smooth and singular part of $X$ by $X_{reg}$ and $X_{sing}$. In particular, $X_{sing}$ is an analytic subvariety of $X$ of complex codimension at least $2$. 

It is shown in \cite{EGZ} that there exists a unique Ricci-flat K\"ahler current  in $c_1(L)$ with bounded local potential.

 Let  $\A$ be an arbitrary ample line bundle on $X'$. Since $X'$ is a Calabi-Yau manifold, there exists  a smooth volume form $\Omega'$ on $X'$satisfying 
$$\ddbar \log \Omega' = 0, ~ \int_{X'} \Omega' = [L']^n.$$ Obviously, $\Omega' = \pi^* \Omega$ for some smooth Calabi-Yau volume form on $X'$.
We can choose $\chi' =\pi^*\chi \in c_1(L')$ be the Fubini-Study metric induced by an projective embedding from the linear system of  $L^k$ for some sufficiently large $k\in \mathcal{F}(X, L)$.  
We also choose $\omega_{\A} \in [\A] $ a fixed smooth K\"ahler metric on $X'$. We then consider the following Monge-Ampere equation
\begin{equation}\label{cy}
(\chi + e^{-t} \omega_{\A} + \ddbar \varphi_t)^n = e^{c_t} \Omega, ~\int_{X'}  \varphi_t dg_t=0~, t\in [0, \infty),  
\end{equation}
where $g_t$ is the K\"ahler metric associated to the K\"ahler form $\omega_t=\chi+ e^{-t} \omega_{\A} + \ddbar \varphi_t$,  $c_t$ is normalization constant satisfying $ e^{c_t} [L']^n =[L'+e^{-t} \A]^n$. It is straightforward to see that  $c_t= O(e^{-t}  )$ for sufficiently large  $t$. 

Equation (\ref{cy}) is solvable for all $t\in [0, \infty)$ by Yau's solution \cite{Y1} to the Calabi conjecture  and $g_t$ is a smooth Ricci-flat K\"ahler metric 
$$Ric(g_t) = 0. $$

\begin{proposition} \label{basie}Let  $\varphi_t$ be the solution of equation (\ref{cy}). There exists $C>0$ such that for all $t\in [0, \infty)$, 
\begin{equation} \label{c0}
||\varphi_t||_{L^\infty(X')} \leq C,
\end{equation}
\begin{equation}\label{schw}
\omega_t \geq C^{-1} \chi', 
\end{equation}
\begin{equation}\label{diame}
diam_{g_t} (X') \leq C,
\end{equation}
where $diam_{g_t} (X')$ is the diameter of $(X', g_t)$. 
Let $E$ be the exceptional locus of the crepant resolution $\pi$, i.e., $E=\pi^{-1} (X_{sing})$. Then for any $k>0$ and $K\subset\subset X'\setminus E$, there exists $C_{k,K}>0$ such that for all $t\in [0, \infty)$
\begin{equation}\label{ck}
||\varphi_t ||_{C^k(K)} \leq C_{k,K}.
\end{equation}

\end{proposition}

\begin{proof} The $C^0$ estimate (\ref{c0}) follows from a general result in \cite{EGZ, Z} and in fact it is shown that it can also be directly obtained by  Moser's iteration \cite{To1, ZY} instead of pluripotential theory. The estimate (\ref{schw}) follows from a Schwarz type estimate \cite{ST1}. The diameter bound is independently derived in \cite{To1, ZY} by directly applying Yau's volume growth for Ricci-flat manifolds \cite{SY}.  The estimate (\ref{ck}) is well-known by a uniform second order estimate for $\varphi_t$ using Tsuji's trick \cite{Ts} and then standard local higher order estimates \cite{PSeS}.

\end{proof}

By letting $t\rightarrow \infty$, we immediately derive the following corollary.

\begin{corollary} \label{limitsol} Let $X_{reg}$ be the regular part of $X$. There exists a unique solution $\varphi_\infty \in PSH(X, \chi)\cap L^\infty(X)  \cap C^\infty(X_{reg})$ solving the equation 
\begin{equation}\label{cy2}
(\chi + \ddbar \varphi_\infty)^n =\Omega.
\end{equation} 
Therefore the K\"ahler form $\omega_\infty=\chi + \ddbar \varphi_\infty$ is the unique Ricci-flat K\"ahler current in $c_1(L)$ with bounded local potential. We let $g_\infty$ be the smooth K\"ahler metric associated to $\omega_\infty$ on $X_{reg}$.

\end{corollary}

In particular, $\varphi_\infty$ is the point-wise limit of $\varphi_t$ as $ t\rightarrow \infty$ by the uniqueness. A natural question is what can be said about the metric completion of $(X_{reg}, g_\infty)$ and how it is related to the original Calabi-Yau variety $X$.


\subsection{A gradient estimate} We will obtain a uniform bound for $|\nabla \varphi_\infty|_{g_\infty}$ for the solution $\varphi_\infty$ with respect to $g_\infty$ in Corollary \ref{limitsol}. For complex Monge-Ampere equations with degenerate or singular data, Holder estimates are of much interest with various applications. Usually such estimates are measured by a fixed background metric, however, in many geometric setting, one should instead treat the gradient estimates intrinsically using the metric induced from the solution.


\begin{lemma} Let $\varphi_t$ be the solution of equation (\ref{cy}). Then $\varphi_t$ is smooth in $t$ for all $t\in [0, \infty)$ and
\begin{equation}
\int_{X'} \dot \varphi_t dV_{g_t}  = 0, 
\end{equation}
where  $\dot \varphi_t=\ddt{\varphi_t}$ and $dV_{g_t} =( \omega_t)^n = e^{c_t} \Omega'$.

\end{lemma}

\begin{proof} The smoothness of $\varphi_t$ in $X'\times [0, \infty)$ follows from the implicit function theorem and 
$$\int_{X'} \dot \varphi_t  dV_{g_t} = \ddt{} \int_{X'} \varphi_t dV_{g_t}- \dot c_t \int_{X'} \varphi_t  dV_{g_t} = 0.  $$

\end{proof}

\begin{lemma} There exists $C>0$ such that for all $t\in [0, \infty)$ 
$$\dot \varphi_t   \leq  C. $$

\end{lemma}

\begin{proof}

Let  $V_t= \int_{X'} \omega_t^n$.  Note that $\dot c_t$ is uniformly bounded and so there exists $C>0$ such that 
$$\Delta_t \dot\varphi_t =  e^{-t}  tr_{\omega_t} (\omega_{\A}) + \dot c_t \geq - C .$$
The Green's function $G_t$ for $(X', g_t)$ is uniformly bounded below for $t\in [0, \infty)$ because the diameter and volume of $(X, g_t)$ are both uniformly bounded from above and below. Therefore
\begin{eqnarray*}
\dot \varphi_t &=&  \frac{1}{V_t} \int_{X'} G_t(x, y) (- \Delta_t \dot \varphi_t ) dV_{g_t}+ \frac{1}{V_t} \int_{X'} \dot \varphi_t dV_{g_t}   \\ 
&= &\frac{1}{V_t} \int_{X'} (G_t(x, y)+A) (-\Delta_t \dot \varphi_t) dV_{g_t}    \\ %
&\leq& CA ,
\end{eqnarray*}
where $G_t$ is bounded below by $-A$ for some fixed constant $A>0$ and  $\Delta_t$ is the Laplace operator with respect to $g_t$.

\end{proof}

The following Schwarz type lemma can be derived  by the same calculations in \cite{ST1} because $Ric(g_t)=0$ and $\varphi_t$ is uniformly bounded in $L^\infty(X)$.

\begin{lemma}
There exists $C>0$ such that for all $t\in [0, \infty)$,
\begin{equation}\label{schest}
\Delta_t ~tr_{\omega_t} (\chi') \geq -C tr_{\omega_t}(\chi')^2  + |\nabla_t tr_{\omega_t}(\chi')|^2. 
\end{equation}

\end{lemma}

We define 
$$u =  \varphi_t + \dot{\varphi}_t. $$ 
Straightforward calculations show that  
\begin{equation}\label{u1}
\Delta_t ~u = n  -  tr_{\omega_t} (  \chi' + e^{-t}\omega_{\A} -e^{-t} \omega_{\A} ) + \dot c_t =-  tr_{\omega_t}(\chi') + n + \dot c_t,  
\end{equation}
\begin{equation}
\Delta_t  ( |\nabla_t u|^2 ) = |\nabla_t \nabla_t u|^2 + |\nabla_t \overline\nabla_t u|^2 -  2 Re \left( \nabla_t ~u \cdot \overline\nabla_t ~tr_{\omega_t} (\chi') \right). 
\end{equation}

\begin{lemma} There exists $C>0$ such that for all $t\in [0, \infty)$ 
$$|u|_{L^\infty(X')} \leq C. $$

\end{lemma}

\begin{proof}

Obviously $u$ is uniformly bounded above from the uniform bounds of $\varphi_t$ and $\dot \varphi_t$. 
From estimate (\ref{u1}), there exists $C>0$ such that $\Delta_t  u \leq C$. 
 We then  have
\begin{eqnarray*}
u &=&  \frac{1}{V_t} \int_{X'} G_t(x, y) (- \Delta_t u ) dV_{g_t} + \frac{1}{V_t} \int_{X'} u dV_{g_t} \\ 
&= &\frac{1}{V_t} \int_{X'} (G_t(x, y)+A) (-\Delta_t u)  dV_{g_t}   \\ %
&\geq& - CA   .
\end{eqnarray*}

\end{proof}

The following lemma follows from straightforward linear estimates for the equation of $u$ and local regularity of $\varphi_t$. 

\begin{lemma}\label{wco} For any $K\subset\subset X_{reg} $ and $k>0$, there exists $C_{k,K}>0$ such that for all $t\in [0, \infty)$, 
$$ |u|_{C^k(K)} \leq C_{k,K}. $$

\end{lemma}

\begin{corollary}\label{c32} 

$\dot \varphi_t$ converges to $0$ in $C^\infty (X_{reg}) $ as $t\rightarrow \infty$.

\end{corollary}

\begin{proof} Using integration by parts, we have for some $C>0$
$$\int_{X'} |\nabla_t \dot \varphi_t|^2 dV_{g_t}  =-\int_{X'} e^{-t} \dot\varphi_t \chi' \wedge \omega_t^{n-1}  + \dot c_t V_t\leq C e^{-t}. $$

By Lemma \ref{wco}, $\dot\varphi$ converges in $C^\infty (X_{reg})$ after passing to any convergent subsequence to a solution $\psi \in  C^\infty (X_{reg})$ satisfying 
$\nabla_\infty \psi = 0$, where $\nabla_\infty$ is the covariant derivative with respect to the limiting metric $g_\infty$ on $X_{reg}$ as in Corollary \ref{limitsol}. Therefore $\psi = 0$ since  $\int_{X'} \psi dV_{g_\infty} = \lim_{t\rightarrow \infty} \int_{X'} \dot\varphi dV_{g_t}= 0$, where $dV_{g_\infty} = (\omega_\infty)^n $.

\end{proof}

\begin{lemma} \label{pregradie}  There exist $A, C>0$ such that for all $t\geq0$, we have on $X'$
\begin{equation} \label{dut}
 |\nabla_t u| ^2 \leq C (A- u). 
 \end{equation}

\end{lemma}

\begin{proof}

We choose $A>0$ such that $ A - u\geq A/2.$ Then for sufficiently large $A$,  
\begin{eqnarray*}
 && \Delta_t \left(  \frac{ |\nabla_t u|^2}{ A- u}     \right)  \\
&\geq & \frac{ |\nabla_t \nabla_t u|^2 + |\nabla_t \overline\nabla_t u|^2- 2 Re\nabla_t u \cdot \overline\nabla_t ( tr_{\omega_t} (\chi') )}{A-u} + \frac{2Re \nabla_t u \cdot \overline \nabla_t |\nabla_t u|^2 }{(A-u)^2} +\frac{2|\nabla_t u|^4}{(A-u)^3}\\
&& + \frac{ |\nabla_t u|^2 (-tr_{\omega_t}(\chi' ) + n + \dot c_t) }{(A-u)^2} \\
&\geq& \frac{ 2|\nabla_t u|^4}{(A-u)^3} - \frac{|\nabla_t u|^2}{(A-u)^2} - \frac{2Re(  \nabla_t u \cdot \overline \nabla_t tr_{\omega_t} (\chi') ) }{A-u} - C
\end{eqnarray*}
holds on $X'\times [0, \infty)$ for some uniform constant $C>0$.
We let $$H = \frac{ |\nabla_t u|^2}{ A- u}  + tr_{\omega_t} (\chi'). $$ Then we immediately  have 
$$\Delta_t  H \geq   \frac{ 2|\nabla_t u|^4}{(A-u)^3} - \frac{ 3 |\nabla_t u|^2}{(A-u)^2} - C $$
for some fixed constant $C>0$ if we choose $A$ sufficiently large. 
By applying the maximum principle, $H$ is uniformly bounded for all $t\in [0, \infty)$ as $tr_{\omega_t}(\chi')$ is uniformly bounded. Then the estimate (\ref{dut}) immediately follows.

\end{proof}

The following proposition is the main result of this section. 

\begin{proposition} \label{gradie} Let $\varphi_\infty$ be the solution in Corollary \ref{limitsol}. There exists $C>0$ such that 
\begin{equation}\label{gradest}
\sup_{X'\setminus E} |\nabla_{g_\infty}  \varphi_\infty|  \leq C . 
\end{equation}

\end{proposition}

\begin{proof}  $\varphi_t$ is uniformly bounded in $L^\infty(X')$ and in local $C^\infty( X'\setminus E)$. In particular,  $\dot \varphi_t$ converges to $0$ in $C^\infty(X'\setminus E)$ by Corollary \ref{c32}.  Let $u_\infty = \lim_{t\rightarrow \infty} u$. 
Then $u_\infty = \varphi_\infty$ and immediately from Lemma \ref{pregradie} $$ |\nabla_{g_\infty} \varphi_\infty|^2 \leq C (A - \varphi_\infty) \leq C' $$
uniformly on $X_{reg}$ for some fixed constant $C$ and $C'$.

\end{proof}


\subsection{$L^2$-estimates}

In this section, we will derive various $L^2$-estimates for global sections in $H^0(X, L^k)$. We first construct a family of cut-off functions along analytic subvarieties. We learn such a construction from J. Sturm \cite{Stu}. 

\begin{lemma} \label{app} Let $Y$ be a  normal projective variety and $Z$ be a subvariety of $Y$ such that $Y\setminus Z$ is smooth. Suppose $\omega$ is a K\"ahler current on $Y$ with bounded local potentials. Then for any $\epsilon>0$ and $K\subset\subset Y\setminus Z $, there exists  $\rho_\epsilon \in C^\infty(Y\setminus Z)$ such that 

\begin{enumerate}

\item $ 0\leq \rho_\epsilon \leq 1$, 
\smallskip

\item $Supp \rho_\epsilon \subset\subset Y\setminus Z$, 
\smallskip

\item $\rho_\epsilon=1$ on $K$, 
\smallskip

\item $ \int_Y |\nabla \rho_\epsilon|^2 \omega^n = \int_{Y_{reg}}  \sqrt{-1}\partial \rho_\epsilon \wedge \dbar \rho_\epsilon \wedge \omega^{n-1}< \epsilon$. 

\end{enumerate}

\end{lemma}

\begin{proof} It suffices to prove for the case when $Y$ is smooth and $Z =D$ is a union of smooth divisors, after possible blow-ups.  Let $s$ be the defining section for $D$ and $h$ be a smooth hermitian metric on the line bundle associated to $[D]$. Without loss of generality, we can assume that $|s|_h^2 \leq 1$. Let $\theta$ be a K\"ahler metric on $Y$ such that $\theta > Ric(h)$ and $[\theta]\geq [\omega]$. 
Let $F$ be the standard smooth cut-off function on $[0, \infty)$ with $F=1$ on $[0, 1/2]$ and $F=0$ on $[1, \infty)$. 
We then let 
$$\eta_\epsilon = \max ( \log |s|^2_h, \log \epsilon) . $$
For sufficiently small $\epsilon$, we have $-\log \epsilon \leq \eta_\epsilon \leq 0$. 
Then obviously, $\eta_\epsilon \in PSH(Y, \theta) \cap C^0(Y)$.
Now we let 
$$\rho_\epsilon = F\left(\frac{\eta_\epsilon}{\log \epsilon} \right).$$
Then  $\rho_\epsilon =1 $ on $K$ if $\epsilon$ is sufficiently small. Straightforward calculations give
\begin{eqnarray*}
&&\int_Y \sqrt{-1} \partial \rho_\epsilon \wedge \dbar \rho_\epsilon \wedge \omega^{n-1}  \\
&=&(\log \epsilon)^{-2}  \int_Y (F')^2 \sqrt{-1}\partial \eta_\epsilon \wedge \dbar \eta_\epsilon \wedge \omega^{n-1}\\
&\leq& C (\log \epsilon)^{-2} \int_Y (-\eta_\epsilon) \ddbar \eta_\epsilon \wedge \omega^{n-1} \\
&\leq& C(\log \epsilon)^{-2} \int_Y(-\eta_\epsilon)(\theta+ \ddbar \eta_\epsilon) \wedge \omega^{n-1}
+ C(\log \epsilon)^{-2} \int_Y \eta_\epsilon ~ \theta \wedge \omega^{n-1}\\
&\leq& C(-\log \epsilon)^{-1} \int_Y(\theta+ \ddbar \eta_\epsilon) \wedge \omega^{n-1} \\
&\leq& C(-\log \epsilon)^{-1} [\theta]^n \rightarrow 0
\end{eqnarray*}
as $\epsilon \rightarrow 0$.
Therefore we obtain  $\rho_\epsilon \in C^0(Y)$ satisfying the conditions in the lemma. The lemma is then proved by smoothing $\rho_\epsilon$ on $Supp~ \rho_\epsilon \setminus K$.

\end{proof}

The rest of the section aims to derive various $L^2$-estimates for holomorphic sections. The difference between our calculations and the standard ones from \cite{T1, DS} is that we directly obtain such estimates on the limiting singular variety instead of the approximating manifolds.  First, we abuse the notation by identifying $\varphi_\infty$ the solution of Corollary \ref{limitsol} on $X$ and $\pi^{-1}(\varphi_\infty)$ on $X'$.  Let $h_{FS}$ be the smooth fixed hermitian metric on $L$ induced from some embedding of $L^k$ for some $k>>1$ such that $Ric(h_{FS}) = \chi$. We define the singular hermitian metric $h_\infty$ and $h_\infty'$ on $L$ and $L'$ by
$$h_\infty = e^{-\varphi_\infty} h_{FS}, ~ h_\infty' = \pi^*h_\infty.$$
Obviously, $Ric (h_\infty) = \omega_\infty$. Without confusion, we identify the above quantities on $(X, L)$ and $(X', L')$. 

The following two lemmas follow immediately from the uniform bound on $\varphi_\infty$. 

\begin{lemma} \label{partial} For any $k\in \mathcal{F}(X, L)$ and a basis $\{ s_j \}_{j=1}^{d_k+1}$ of $H^0(X, L^k)$, there exists $\epsilon>0$ such that 
\begin{equation}
\inf_{z \in X_{reg}} \sup_{j=1, ..., d_k+1} |s_j(z) |_{h_\infty^k} ^2 \geq \epsilon. 
\end{equation}

\end{lemma}

\begin{proof} This immediately follows from the choice of $k$ so that $L^k$ is globally generated and the fact that $\varphi_\infty$ is uniformly bounded in $L^\infty(X)$.

\end{proof}

By the boundedness of $\varphi_\infty$, we have the following lemma.

\begin{lemma} \label{l20} For any $s\in H^0(X, L^k)$, there exists $C_{k, s}>0$ such that 
$$\sup_{X} |s|_{h_\infty^k} ^2 \leq C_{k,s}. $$

\end{lemma}

We define the scaled norm $|| \cdot ||_{L^{\infty, \sharp}}$ and $|| \cdot ||_{L^{2, \sharp}}$ for $s\in H^0(X, L^k)$ with respect to the hermitian metric $h_\infty^k$ and $kg_\infty$.

\begin{proposition} \label{l21}  There exists  $K >0$ such that if  $s\in H^0(X, L^k)$ for $k\geq 1$, then 
\begin{equation}\label{l1}
\|s\|_{L^{\infty, \sharp}} \leq K \|s\|_{L^{2, \sharp}}.
\end{equation}

\end{proposition}

\begin{proof} We break the proof into the following steps.

\begin{enumerate}

\item We write $|\nabla s|= |\nabla s|_{h_\infty^k, k g_\infty}$ and $|s|=|s|_{h_\infty^k}$ for simplicity. Straightforward local calculations show that on $X_{reg} $ 
$$ \Delta |s| \geq - |s|,  $$
where $\Delta$ is the Laplacian operator with respect to $k g_\infty$.  
We can lift $s$ and other quantities from $X$ to $X'$. Then for any $p\in \mathbb{Z}^+$, 
 $$\int_{X'} \rho_\epsilon^2 |s|^p \Delta (-|s|) dV_{kg_\infty} \leq  \int_{X'} |s|^{p+1} dV_{kg_\infty},   $$
where $\rho_\epsilon$ is constructed from Lemma \ref{app} for $Y=X'$ and $Z=E$. Also 
\begin{eqnarray*}
&& \int_{X'} \rho_\epsilon^2 |s|^p \Delta(-|s|)  dV_{kg_\infty}  \\
&=& \frac{4p}{(p+1)^2}\int_{X'} \rho_\epsilon^2 |\nabla(  |s|^{(p+1)/2} )|^2 dV_{kg_\infty}  - 2\int_{X'} |s|^p \rho_\epsilon \nabla \rho_\epsilon \cdot \nabla |s| dV_{kg_\infty} \\
&\geq & \frac{4p}{(p+1)^2}\int_{X'} \rho_\epsilon^2 |\nabla(  |s|^{(p+1)/2} |^2 dV_{kg_\infty}  \\
&& +2 (A_k)^{(p+1)/2} \left(\int_{X'}  | \nabla \rho_\epsilon|^2 \right)^{1/2} \left( \int_{X'}  \rho_\epsilon^2\left| \nabla |s|^{(p+1)/2} \right|^2 \right)^{1/2} dV_{kg_\infty} ,
\end{eqnarray*}
where $A_k= \sup_{X_{reg}} |s|<\infty$. 
Let $\epsilon \rightarrow 0$, we have

$$ \int_{X'} |\nabla |s|^{(p+1)/2}|^2 dV_{kg_\infty} \leq  p \int_{X'} |s|^{p+1} dV_{kg_\infty} .$$

\item  We would like to prove a Sobolev type inequality for $g_\infty$.

\begin{claim} There exists $K>0$ such that for all $f \in L^\infty (X') \cap L^{1, 2}(X', k\omega_\infty)$.  
$$|| \nabla f ||_{L^2(X', k\omega_\infty)}  \geq C_S || f ||_{L^{\frac{2n}{n-1}}(X', k \omega_\infty)} - || f ||_{L^2(X', k \omega_\infty)} .$$

\end{claim}

\begin{proof}

First we define $f_\epsilon = \rho_\epsilon f$, where $\rho_\epsilon$ is constructed from Lemma \ref{app} for $Y=X'$ and $Z=E$.  Then for fixed $\epsilon>0$, 
$$||\nabla _{g_t} f_\epsilon ||_{L^2(X', k \omega_t )} \rightarrow ||\nabla f_\epsilon ||_{L^2(X', k \omega_\infty)}. $$
as $t\rightarrow \infty$ by Proposition \ref{basie}. The Sobolve constant is uniformly bounded for $(X', g_t)$ because of the uniform bound on the Ricci curvature, volume and diameter of $(X, g_t)$. Thus the Sobolev constant is also uniformly bounded for $(X', kg_t)$ $k=1, 2, ...$ and so there exists $K>0$ such that  for all $t$, 
$$ ||\nabla_{g_t} f_\epsilon ||_{L^2(X', k \omega_t)} \geq C_S ||f_\epsilon||_{L^{2n/(n-1)}(X',  k \omega_t)} - ||f_\epsilon ||_{L^2(X', k \omega_t)}. $$
By letting $t \rightarrow \infty$, we have 
$$ ||\nabla f_\epsilon ||_{L^2(X', k \omega_\infty)} \geq C_S ||f_\epsilon||_{L^{2n/(n-1)}(X', k \omega_\infty )} - ||f_\epsilon ||_{L^2(X', k \omega_\infty )}. $$
Notice that by letting $\epsilon \rightarrow 0$, we have 
$$ ||f_\epsilon||_{L^{2n/(n-1)}(X', k \omega_\infty)} \rightarrow ||f ||_{L^{2n/(n-1}(X', k \omega_\infty)}, ~~~ ||f_\epsilon ||_{L^2(X', k \omega_\infty )} \rightarrow ||f  ||_{L^2(X', k \omega_\infty)}. $$

It now suffices to show that $ ||\nabla f_\epsilon ||_{L^2(X', k \omega)} \rightarrow  ||\nabla f  ||_{L^2(X', k \omega)}$ as $\epsilon \rightarrow 0$. This follows from the following calculations.
\begin{eqnarray*}
 &&\left| \int_{X'} |\nabla f|^2dV_{kg_\infty}   - \int_{X'} |\nabla f_\epsilon|^2 dV_{kg_\infty} \right|  \\
&=& \left|  \int_{X'} |\nabla f|^2 dV_{kg_\infty}  - \int_{X'}  |\rho_\epsilon \nabla f   + f \nabla \rho_\epsilon|^2  dV_{kg_\infty} \right| \\
&\leq & \int_{X'} (1- \rho_\epsilon^2) |\nabla f|^2 dV_{kg_\infty} + \int_{X'} f^2 |\nabla \rho_\epsilon|^2 dV_{kg_\infty} \rightarrow  0
\end{eqnarray*}
since $f$ is bounded.

\end{proof}

\item Now we can apply the standard Moser's iteration and complete the proof of the lemma.

\end{enumerate}

\end{proof}

The following lemma gives a point-wise bound for the gradient of $s\in H^0(X, L^k)$.

\begin{lemma}\label{l200} For any $s\in H^0(X, L^k)$, there exists $C_k>0$ such that 
$$\sup_{X_{reg}} |\nabla s|_{h_\infty^k, g_\infty}^2 \leq C_k. $$

\end{lemma}

\begin{proof}  Since $h_\infty = e^{-\varphi_\infty} h_{FS}$, we have 
$$\nabla s = \partial s +  k s ~\partial \varphi_\infty-  k s ~\partial \log h_{FS}. $$
This implies that 
\begin{eqnarray*}
|\nabla s|_{h_\infty^k, g_\infty}   &\leq&  |\nabla_{h_{FS}^k}  s |_{h_\infty^k, g_\infty}  + k |s|_{ h_\infty^k} |\nabla \varphi_\infty |_{g_\infty} \\
&\leq& C  |\nabla_{(h_{FS})^k}  s |_{ (h_{FS})^k , g_{FS}}  + k |s|_{h_\infty^k}  |\nabla \varphi_\infty |_{g_\infty}
\end{eqnarray*}
for some $C>0$, where $g_{FS}$ is the metric associated to the K\"ahler form $\chi$. The last inequality follows from the estimate 
$$\omega_\infty \geq C^{-1} \chi $$
on 
$X_{reg}$ for some uniform $C>0$ by applying the estimate (\ref{schw}) in Proposition \ref{basie} and by  letting $t\rightarrow \infty$. 

Since $s$ is the restriction of a hyperplane section on $\mathbb{CP}^{d_k}$, $\nabla_{h_{FS}^k }s = (\tilde \nabla_{h_{FS}^k} s ) |_{X}$ is the restriction of the $L^k$-valued one form on $\mathbb{CP}^{d_k}$ to $X$, where $\tilde\nabla_{h_{FS}}$ is the connection of $h_{FS}^k$ on the hyperplane bundle of $\mathbb{CP}^{d_k}$.  Hence $|\nabla_{h_{FS}^k}s|_{(h_{FS})^k, g_{FS}}$ is bounded after being restricted from $\mathbb{CP}^{d_k}$ to $X$. On the other hand, from the gradient estimate for $\varphi_\infty$ in Proposition \ref{gradie} and Lemma  \ref{l20}, $|s|_{h_\infty^k}  |\nabla \varphi_\infty |_{g_\infty}$ is also bounded. This completes the proof of the lemma. 

\end{proof}

We now prove a uniform version for Lemma \ref{l200}.

\begin{proposition}  \label{l22} There exists  $K >0$ such that if  $s\in H^0(X, L^k)$ for $k\geq 1$, then 
  $$ \|\nabla s\|_{L^{\infty, \sharp}} \leq K \|s\|_{L^{2, \sharp}} $$

\end{proposition}

\begin{proof} We break the proof into the following steps.

\begin{enumerate}

\item By straightforward calculations, we have 
\begin{eqnarray}
\Delta |\nabla s|^2 
&=& |\nabla\nabla s|^2 - 2|\nabla s|^2 + n|s|^2
\\
\Delta |\nabla s|
&=& - |\nabla s| + \frac{n}{2} \frac{|s|^2}{|\nabla s|}  + \frac{|\nabla \nabla s|^2}{2|\nabla s|} - \frac{|\nabla |\nabla s|^2 |^2}{4 |\nabla s|^3},
\nonumber
\end{eqnarray}
where we write $|\nabla s|= |\nabla s|_{h_\infty^k, k g_\infty}$ and $|s|=|s|_{h_\infty^k}$ for simplicity. 
On the other hand, we have the following inequality
\begin{eqnarray}
|\nabla |\nabla s|^2 |^2 
&=& g^{k\bar\ell}g^{i\bar j}g^{p\bar q}
(\nabla_i\nabla_k s\nabla_{\bar q}\nabla_{\bar \ell}\bar s\nabla_ps\nabla_{\bar j}\bar s
+
g_{p\bar\ell}\nabla_i\nabla_k s\nabla_{\bar j}\nabla_{\bar q}\bar s
\nonumber
\\
&&
\qquad\qquad\quad
+g_{k\bar j}\nabla_{\bar q}\nabla_{\bar\ell}\bar s
\nabla_i s\nabla_p s
+
g_{k\bar j}g_{p\bar\ell}|s|^2
\nabla_is\nabla_{\bar q}s)\nonumber
\\
&\leq&|\nabla\nabla s|^2 |\nabla s|^2 + 2 |s||\nabla s|^2 |\nabla\nabla s| + |s|^2|\nabla s|^2,  
\end{eqnarray}
where we write $g=g_\infty$ for simplicity.
Using this inequality, we obtain  
$$\Delta |\nabla s| \geq -  \frac{1}{2} |\nabla s| .$$

\item  Integration by parts gives
\begin{eqnarray*}
&& \frac{1}{2} \int_{X'} (\rho_\epsilon)^2|\nabla s|^{p+1} dV_{k g_\infty} \\
&\geq& \int_{X'} (\rho_\epsilon)^2 |\nabla s  |^p \Delta(-|\nabla s|) dV_{kg_\infty}\\
&=& \frac{4p}{(p+1)^2}\int_{X'} (\rho_\epsilon)^2 |\nabla(  |\nabla s|^{(p+1)/2} )|^2 dV_{kg_\infty} - 2 \int_{X'} |s|^p \rho_\epsilon \nabla \rho_\epsilon \cdot \nabla |\nabla s| dV_{kg_\infty} \\
&\geq & \frac{4p}{(p+1)^2}\int_{X'} (\rho_\epsilon)^2 |\nabla(  |\nabla s|^{(p+1)/2}) |^2 dV_{kg_\infty} \\
&&- (A_k)^{(p+1)/2} \left(\int_{X'}  | \nabla \rho_\epsilon|^2 dV_{kg_\infty}\right)^{1/2} \left( \int_{X'}  ( \rho_\epsilon)^2 \left| \nabla |\nabla s|^{(p+1)/2} ) \right|^2 dV_{kg_\infty} \right)^{1/2}, 
\end{eqnarray*}
where $A_k = \sup_{X_{reg}} |s| < \infty$ and  $\rho_\epsilon$ is constructed from Lemma \ref{app} for $Y=X'$ and $Z=E$. 
Since $\int_{X'}  |\nabla \rho_\epsilon|^2 dV_{kg_\infty} \rightarrow 0 $ as $\epsilon \rightarrow 0$, we  have 
$$ \int_{X'} |\nabla |\nabla s|^{(p+1)/2}|^2 dV_{kg_\infty} \leq  p \int_{X'} | \nabla s|^{p+1} dV_{kg_\infty}$$
after letting $\epsilon \rightarrow 0$.

\medskip

\item Since $|\nabla s|^2$ is bounded, by the same argument using the cut-off function, we can apply Moser's iteration to the following estimates
$$ \int_{X'} |\nabla |\nabla s |^{(p+1)/2} |^2 dV_{kg_\infty} \leq  p \int_{X'} |\nabla s |^{p+1}dV_{kg_\infty} $$
$$ ||\nabla |\nabla s|^{(p+1)/2} ||_{L^{2, \sharp} }\geq C_S || |\nabla s|^{(p+1)/2}||_{L^{2n/(n-1), \sharp} } - || |\nabla s|^{(p+1)/2}||_{L^{2, \sharp} }. $$
We then can complete  the proof of the proposition by standard argument.

\end{enumerate}

\end{proof}

We need the following version of $L^2$-estimates due to Demailly   (Theorem 3.1 in \cite{De}) for a big and net line bundle over a projective manifold. 

\begin{theorem} \label{demailly} Suppose that $X$ is an $n$-dimensional  projective manifold equipped with a smooth K\"ahler metric $\omega$. Let $L$ be a holomorphic line bundle over $X$ equipped with a possibly singular hermitian metric $h$ such that $Ric(h) = -\ddbar \log h \geq \delta \omega$ in current sense for some $\delta>0$. Then for every  $L$-valued $(n,1)$-form $\tau$ satisfying 
$$ \dbar \tau =0, ~ \int_X |\tau|^2_{h, \omega} ~\omega^n<\infty,$$
where $|\tau|_{h, \omega}^2 = tr_{\omega} \left( \frac{ h \tau \overline\tau}{\omega^n} \right)$, 
there exists an  $L$-valued $(n, 0)$-form $u$ such that $\dbar u = \tau$ and 
\begin{equation}
\int_X |u|^2_h ~\omega^n \leq \frac{1}{2\pi \delta} \int_X |\tau|_{h, \omega}^2 ~\omega^n.
\end{equation}

\end{theorem}

We will apply Theorem \ref{demailly} to obtain the following proposition.

\begin{proposition} \label{l3} Let $X$ be a projective Calabi-Yau variety with crepant singularities. If $L$ is an ample line bundle over $X$ equipped with a hermitian metric such that $\omega=Ric(h) \in c_1(L) $ is the unique Ricci-flat K\"ahler current on $X$ with bounded local potentials, then for any smooth $L$-valued $(0,1)$-form $\tau$ satisfying

\begin{enumerate}

\item $\dbar \tau =0$,  

\item $Supp ~\tau \subset \subset X_{reg}$, 

\end{enumerate}
there exists an $L$-valued section $u$ such that $\dbar u = \tau$ and $$ \int_X |u|^2_h ~\omega^n \leq \frac{1}{2\pi} \int_X |\tau|^2_h~ \omega^n. $$

\end{proposition}

\begin{proof}  Let $\pi: X' \rightarrow X$ be the crepant resolution of $X$.   Since $L$ is ample on $X$, by Kodaira's lemma, there exists a divisor $D$ on $X'$ such that $L' - \epsilon [D]$ is ample for all $\epsilon>0$, where $L'= \pi^*L$.    Let $s_D$ be the defining section of $D$ and $h_D$ a smooth hermitian metric on the line bundle induced by $[D]$ satisfying $$\chi + \epsilon \ddbar \log h_D >0$$ 
for all sufficiently small $\epsilon>0$, where $\chi \in c_1(L')$ is a smooth closed semi-positive $(1,1)$-form as the pullback of the Fubini-Study metric from the linear system of $|(L')^k|$ for some sufficiently large $k$. We consider the following Monge-Ampere equation 
$$(\chi + \epsilon \ddbar \log h_D + \ddbar \varphi_\epsilon)^n = c_\epsilon \Omega, ~ \int_{X'} \varphi_\epsilon \Omega =0 $$
where $\Omega$ is a fixed smooth Calabi-Yau volume form on $X'$ and $c_\epsilon \int_{X'}\Omega = [L - \epsilon D]^n \rightarrow [L^n]$ as $\epsilon \rightarrow 0$. 
Obviously  $$\omega_\epsilon = \chi + \epsilon \ddbar \log h_D + \ddbar \varphi_\epsilon $$ is the unique Ricci flat K\"ahler metric in $[L'-\epsilon D]$. $\varphi_\epsilon$, $h_\epsilon$ and $\omega_\epsilon$ converges to $\varphi$, $h=h_{FS}e^{-\varphi}$ and $\omega=\chi + \ddbar \varphi$ weakly globally on $X$ and smoothly on $X'\setminus D$ as $\epsilon \rightarrow 0$.

We can identify $s\in H^0(X', L')$ and $s \in H^0(X', (L' -K_{X'})+  K_{X'})$ and the hermitian metrics on $L'$ are also hermitian metrics on $L'-K_{X'}$ because $K_{X'}$ is numerically trivial.

We now define $h_\epsilon = h_{FS} e^{ - \epsilon \log |s_D|_{h_D}^2 - \varphi_\epsilon}$ to be a hermitian metric on $L'$, in particular, 
$$ Ric(h_\epsilon)= -\ddbar \log h_{\epsilon} = \omega_\epsilon + [s_D] \geq \omega_\epsilon$$
in the sense of currents.  Hence we can apply Theorem \ref{demailly} to $X', L', \omega_\epsilon, h_\epsilon, \tau' = \pi^*\tau $. Note that  $\tau$ is smooth and 
$$\lim_{\epsilon\rightarrow 0} \int_{X'} |\tau'|_{h_\epsilon}^2 \omega_\epsilon^n =\int_{X'} |\tau'|_{h}^2 \omega^n < \infty$$
since $\tau'$ vanishes in a neighborhood of the exceptional locus of $\pi$ and so $h_\epsilon$ and $\omega_\epsilon$ converges smoothly on the support of $\tau'$.
By Theorem \ref{demailly}, there exists $u_\epsilon $ on $X'$ such that 
$$\dbar u_\epsilon = \tau', ~~ \int_{X'} |u_{\epsilon}|^2_{h_\epsilon} \omega_\epsilon^n \leq \frac{1}{2\pi} \int_{X'} |\tau'|_{h_\epsilon}^2 \omega_\epsilon^n. $$

%
%

Also $h$ and $h_{FS}$ are uniformly equivalent since $\varphi$ is uniformly bounded, $$\int_{X'} |u_\epsilon|^2_h~ \omega^n \leq C \int_{X'} |u_\epsilon|^2_{h_{FS}} \Omega \leq C \int_{X'} |u_\epsilon|^2_{h_\epsilon} ~\omega_\epsilon^n $$
is uniformly bounded for all $\epsilon>0$. 

Hence we can take a subsequence of $u_{\epsilon_v} $ converging weakly in $L^2(X', \Omega)$ to $u \in L^2(X, \Omega)$ as $\epsilon_v \rightarrow 0$ and 
$$\dbar u =\dbar \lim_{\epsilon \rightarrow 0} u_{\epsilon_v} = \tau'$$
on $X'$.
On the other hand, $u$ is bounded uniformly in $L^2(X, h_{FS}e^{-\epsilon \log |s_D|^2_{h_D}} \Omega)$ for any $\epsilon>0$ and $\omega_\epsilon$ converges to $\omega$ in $C^\infty(X, D)$,  hence  
$$\int_{X'} |u|^2_h \omega^n \leq \int_{X'} |\tau' |_h^2  \omega^n.  $$
The proof is complete after pushing $u$ to $X$.
\end{proof}

\medskip


\subsection{Gromov-Hausdorff limits}

In this section, we use the standard theory of Cheeger-Colding to construct the metric completion of $(X_{reg}, g_\infty)$ as a metric length space as the Gromov-Hausdorff limit of $(X', g_t)$ as $t\rightarrow \infty$.

\begin{proposition} \label{ghlim} Let $g_t$ be the unique Ricci-flat K\"ahler metric solving equation (\ref{cy}) on $X'$. Then $(X', g_t)$ converges in Gromov-Hausdorff topology to a unique metric length space $(X_\infty, d_\infty)$. Furthermore,  

\smallskip

\begin{enumerate}

\item $(X_\infty, d_\infty)$ is the metric completion of $(X_{reg}, g_\infty)$.

\smallskip

\item $X_\infty = \mathcal{R}  \cup \mathcal{S} $, where $ \R$ is the regular part  of $(X_\infty, d_\infty)$ and $\Sc$ is the singular set. In particular, $\Sc$ is closed and its Hausdorff dimension is no greater than $n-4$. 

\smallskip

\item $\R$  is  convex in $(X_\infty, d_\infty)$ and 
\begin{equation}\label{reguset}
\R= X_{reg}. 
\end{equation}

\end{enumerate}

\end{proposition}

\begin{proof} The Gromov-Hausdorff convergence and (2) are standard results from Cheeger-Colding theory \cite{CC1, CC2, CCT}. The statement (1) and (\ref{reguset})  are due to Rong-Zhang \cite{RZ} and the convexity of $(X_\infty)_{reg}$ is proved in \cite{CN}. 

\end{proof}

The fact that $\R=X_{reg}$ is quite important for our later constructions for local sections when applying the $H$-condition from \cite{DS}. Proposition \ref{ghlim} also allows us to  identify $X_{reg}$ as an open dense set of $X_\infty$. 

\begin{lemma} \label{extenm} Any holomorphic section $\sigma \in H^0(X, L^k)$ can be continuously extended from $X_{reg}$ to $X_\infty$. 

\end{lemma}

\begin{proof} For any $\sigma \in H^0(X, L^k)$,  both $|\sigma|_{h_\infty^k}^2$ and $|\nabla_{g_\infty} \sigma |_{h_\infty^k}^2$  are uniformly bounded on $X_{reg}$ by Lemma \ref{l20} and Lemma \ref{l200}. This implies that $\sigma$ is uniformly Lipschitz with respect to $g_\infty$ and $h_\infty$, while $h_\infty$ is equivalent to $h_{FS}$. Hence $\sigma$ can be uniquely extended to $X_\infty$ since $X_{reg}$ is open dense in $X_\infty$ and the metric completion of $(X_{reg}, g_\infty)$ is  $(X_\infty, d_\infty)$.

\end{proof}

Now we can use $H^0(X, L^k)$ to construct maps from $X_\infty $ to projective spaces. Let $d_k +1= \dim H^0(X,L^k)$. We  define
$$\Phi_{k, \sigma} : X_\infty \rightarrow \mathbb{CP}^{d_k} $$
by 
$$\Phi_{k, \sigma} (z) = [ \sigma_1(z), ..., \sigma_{d_k+1}(z) ]$$
where $\{\sigma_1, ..., \sigma_{d_k+1} \}$ is a basis of $H^0(X, L^k)$. Then we immediately have the following corollary.

\begin{corollary} \label{24} $\Phi_{k, \sigma}: X_\infty \rightarrow X \subset \mathbb{CP}^{d_k}$ is a Lipschitz surjective map for   sufficiently large $k\in \mathcal{F}(X,L)$ with respect to $g_\infty$ and $h_\infty$ on $X_\infty$.

\end{corollary} 

\begin{proof}  $\Phi_{k, \sigma}$ is well-defined by Lemma \ref{extenm} and Lemma \ref{partial}. It is also  continuous on $X_\infty$ from Lemma \ref{extenm} and the fact that $X_{reg}$ is an open convex dense in $X_\infty$. Obviously $\Phi_{k, \sigma}$ is Lipschitz, since $g_\infty$ is uniformly bounded below by a multiple of the Fubini study metric on $X_{reg}$ by Proposition \ref{basie} or directly by Lemma \ref{l200}. For any point $z\in X_{reg}$, there exists $p\in X_{\infty}$ with $\Phi_{k, \sigma}(p)=z$ for sufficiently large $k$ since $\Phi_{k, \sigma}$ is an isomorphism on $X_{reg}$. $\Phi_{k, \sigma}$ is then surjective because $\Phi_{k, \sigma}$ is continuous and $X_\infty$ is the metric completion of $X_{reg}$ with $g_\infty$ bounded below by a multiple of $\chi$.

\end{proof}

Since $L$ is semi-ample, $\Phi_{k, \sigma}$ is stabilized for sufficiently large $k\in \mathcal{F}(X,L)$. The following corollary immediately follows from the continuity of $\Phi_{k, \sigma}$.

\begin{corollary} \label{25}There exists a surjective Lipschitz map $\Phi: X_\infty \rightarrow X$ such that 
$$\Phi_{k, \sigma} = \Phi $$
for $k \in \mathcal{F}(X, L)$ sufficiently large. In particular, $\Phi|_{\R}= id$ by identifying $\R=X_{reg}$.

\end{corollary}

In fact, one can just identify $\Phi$ to be the unique continuation of the identity map on $X_{reg}$ through metric completion of $g_\infty$ for the domain and $\chi$ for the target. The goal is to show that $\Phi$ is a homeomorphism and so we have to prove that $\Phi$ is injective.


\medskip

\subsection{$H$-condition and local sections}

We now consider  the $H$-property introduced by Donaldson-Sun \cite{DS}. 

\begin{definition} We consider the follow data $(p_*, D, U, J, L , g, h, A)$ satisfying 

\begin{enumerate}

\item  $(p_*, U, J, g)$  is an open bounded K\"ahler manifold with a complex structure $J$, a K\"ahler metric $g$ and a base point $p_* \in D\subset \subset U$ for an open set $D$, 

\item $L \rightarrow U$ is a hermitian line bundle equipped with a hermitian metric $h$ and $A$ is the connection induced by the hermitian metric $h$ on $L$, with  its curvature $\Omega(A) = g$.

\end{enumerate}

\noindent The data $(p_*, D, U, L, J, g, h, A)$ is said to satisfy the  $H$-condition if there exist $C>0$ and  a compactly supported smooth section $\sigma: U\rightarrow L$   satisfying 

\begin{enumerate}

\item [$H_1$:]   $\|\sigma \|_{L^2(U)} < (2\pi)^{n/2}$,

\item [$H_2$:]  $|\sigma(p_*) | >3/4$,

\item [$H_3$:] for any holomorphic section $\tau$ of $L$ over a neighborhood of $\overline{D}$,
$$|\tau(p_*)| \leq C (\| \dbar \tau \|_{L^{2n+1}(D)} + ||\tau||_{L^2(D)} ),$$

\item [$H_4$:]  $\|\dbar \sigma \|_{L^2(U)} < \min \left( \frac{1}{8\sqrt{2} C},  10^{-20} \right)$,

\item [$H_5$:]  $||\dbar \sigma ||_{L^{2n+1}(D)} \leq \frac{1}{8C}$.

\end{enumerate}

\end{definition}

Here all the norms are taken with respect to $h$ and $g$. The constant $C$ in the $H$ condition depends on the choice $(p_*, D, U, J, L, g, h)$. 

Fix any point $p$ on $X$,  $(X, p, k\omega)$ converges in pointed Gromov-Hausdorff topology to a tangent cone $C(Y)$ over the cross section $Y$. We still use $p$ for the vertex of $C(Y)$.  We write $Y_{reg}$ and $Y_{sing}$ the regular and singular part of $Y$. $Y_{sing}$ has Hausdorff dimension strictly less than $2n-3$.  $C(Y_{reg})\setminus \{p \}$ has a natural complex structure induced from the Gromov-Hausdorff limit and the cone metric $g_C$ on $C(Y)$ is given by
$$g_C = \frac{1}{2} \ddbar r^2 ,$$
where $r$ is the distance function for any point $z\in C(Y)$ to $p$. We can also write the cone metric $g_C= \frac{1}{2} \ddbar |z|^2$.

We need the following proposition due to \cite{DS} to construct  special local sections of $L_0$. 

\begin{proposition} \label{cut2} For any $\epsilon>0$, there exists a cut-off function $\rho$ on $Y$ such that 

\begin{enumerate}

\item $ \rho \in C^\infty(Y_{reg})$,  $0\leq \rho \leq 1$,  
\smallskip

\item $\rho $  is supported in the $\epsilon$-neighborhood of $Y_{sing}$,  

\smallskip

\item $\rho=1$ on a neighborhood of $Y_{sing}$,

\smallskip

\item $|| \nabla \rho||_{L^2(Y, g_C)} < \epsilon. $

\end{enumerate}

\end{proposition}

One considers the trivial line bundle $L_C$ on $C(Y)$ equipped with the connection $A_C$ whose curvature coincides with $g_C$. The curvature of the hermitian metric defined by $h_C=e^{-|z|^2}$ is $g_C$.  $1$ is a global section of $L_C$ with its norm equal to  $e^{-|z|^2}$ with respect to $h_C$. 
The following lemma is due to \cite{DS}. 

\begin{lemma} \label{hcon} Let $p_* \in C(Y_{reg})$. If $3/4 <e^{ - |p_*|^2}  <1$, then for any $\epsilon>0$, there exists $U\subset \subset C(Y_{reg})\setminus \{p\}$ and an open neighborhood $D\subset\subset U$ of $p_*$ such that 
 $(p_*, D, U, L_C, J_C, g_C, h_C, A_C)$ satisfies the $H$-condition.

\end{lemma}


%
%
%
%
%
%
%








From the construction in Proposition \ref{cut2} and Lemma \ref{hcon}, we can always assume that both $D$ and $U$ are a product in $C(Y_{reg})\setminus \{p \}$, i.e., there exist $D_Y$ and $U_Y \subset Y_{reg}$ such that $D=\{ z=(y, r) \in C(Y)~|~y \in D_Y, ~r\in (r_D, R_D)\}$ and  $U=\{ z=(y, r) \in C(Y)~|~y \in U_Y, ~r\in (r_U, R_U)\}$. Suppose $(p_*, D, U, L_C, J_C, g_C, h_C , A_C)$ satisfies the $H$-condition from Lemma \ref{hcon}. For any $m\in \mathbb{Z}^+$, we can define 
\begin{equation}\label{rescue}
U(m) = \{ z=(y, r)\in C(Y)~|~y\in U_Y, ~ r\in (m^{-1/2} r_U, m^{1/2} R_U) \} 
\end{equation}
and
$\mu_m: U \rightarrow U(m)$ by
$$ \mu_m(z) = m^{-1/2} z. $$

The following proposition from \cite{DS} establishes the stability of the $H$-condition for perturbation of the curvature and the complex structure. 

\begin{proposition}  \label{hcon2} Suppose $(p_*, D, U, J_C, L_C, g_C, h_C, A_C )$ constructed as above in Lemma \ref{hcon} satisfies the $H$-condition. There exist $\epsilon>0$ and $m \in \mathbb{Z^+}$ such that for any collection of data $(p_*, D, U, J, g, h, A)$ if 
$$|| g - g_C ||_{C^0(U(m)) } + ||J-J_C ||_{C^0(U(m)) } < \epsilon, $$
then for some $1\leq l \leq m$, 
$$(p_*, D, U, \mu_l^*J, \mu_l^*L, \mu_l^*g, \mu_l^*h, \mu_l^*A)  $$
satisfies the $H$-condition.

\end{proposition}


\bigskip

\subsection{Separating points}

Fix any point $p$, we can assume that $(X_\infty, p, k_v g_\infty)$ converges to a tangent cone $C(Y)$ for some sequence $k_v$ in pointed Gromov-Hausdorff topology. In particular, on the regular part of $C(Y)$, the convergence is locally $C^{2, \alpha}$ and the metrics $k_v g_\infty$ converge locally in $C^{1, \alpha}$.  Fix any open set $U\subset\subset C(Y_{reg})\setminus \{p\}$, This  would induce embeddings $\chi_{k_v}: U \rightarrow \mathcal{R}= (X_\infty)_{reg}$. Let $g_{k_v}$ be the pullback metric of $g_\infty$ on $(X_\infty)_{reg}$ and $J_{k_v}$ be the pullback complex structure. The following lemma follows from the convergence of $(X_\infty, p, k_v g_\infty)$.

\begin{lemma} \label{closen} There exists $v$ such that one can find an embedding $\chi_{k_v}$such that 

\begin{enumerate}

\item $2^{-1} |z| \leq (k_v)^{1/2} d_\infty(p, \chi_{k_v}(z)) \leq 2   |z|$,

\smallskip 

\item $|| \chi_{k_v}^*(k_v g_\infty ) - g_C ||_{C^0(U)} + ||\chi^*_{k_v} J_\infty - J_0||_{C^0(U)} \leq \epsilon $,
\smallskip

\end{enumerate}
where $d_\infty$ is the metric on $X_\infty$ induced from $g_\infty$.

\end{lemma}

\begin{proposition}\label{inject}   For any two distinct points $p$ and $q$ in $X_\infty$, 
$$\Phi (p) \neq \Phi (q). $$

\end{proposition}

\begin{proof} We break the proof into the following steps. 

\begin{enumerate}

\item Suppose $C(Y_p)$ and $C(Y_q)$ are two tangent cones of $p$ and $q$ on $X_{\infty}$ after rescaling $(X_\infty, g_\infty)$ at $p$ by $k_{v_p} \rightarrow \infty $ and at $q$ by $k_{v_q} \rightarrow\infty $. One then can construct two collections of data 
$(p_*, D_p, U_p, J_p, g_p, A_p)$ on $C(Y_p)$ and $(q_*, D_q, U_q, J_q, g_q, A_q)$ on $C(Y_q)$ from Lemma \ref{hcon}, satisfying the $H$-condition. 
In addition, we can always assume that 
\medskip

\begin{itemize}
\item  $ d_{C(Y_p)} (p_*, p) = r_{p_*}  \leq (100K)^{-1}, ~~  d_{C(Y_q)} (q_*, q) = r_{q_*}  \leq (100K)^{-1}, $
where the constant $K$ is defined in Proposition \ref{l21} and Proposition \ref{l22}. 

\medskip

\item the constants $C$ in the $H$-condition for $U_p$ and $U_q$ are the same,  

\medskip

\item $k_{v_p} = k_{v_q}=k_{v_{p,q}}$.

\end{itemize}

\medskip

\item 
From Lemma \ref{closen}, there exist $\chi_p: U_p(2m_p)  \rightarrow  X_\infty $ and $\chi_q: U_q(2m_q) \rightarrow X_\infty$ such that 
\medskip

\begin{itemize}

\item $2^{-1} |z| \leq (k_{v_{p,q}})^{1/2} d_\infty(p, \chi_p(z)) \leq 2   |z|$,

\medskip

\item  $2^{-1} |z| \leq (k_{v_{p,q}})^{1/2} d_\infty(q, \chi_q(z)) \leq 2   |z|$,
\medskip

\item $|| \chi_p^*(k_{v_{p,q}} g_\infty ) - g_C||_{C^0(U_p(m_p))} + || \chi_p^* J_\infty - J_C||_{C^0(U_p (m_p ))} \leq \epsilon $, 

\medskip

\item 
$|| \chi_q^*( k_{v_{p,q}} g_\infty ) - g_C ||_{C^0(U_q(m_q))} + || \chi_q^* J_\infty - J_C ||_{C^0(U_q(m_q))} \leq \epsilon $, 
\medskip

\item $\chi_p(U_p) \cap \chi_q(U_q) = \phi$,  
\medskip
\end{itemize}
where $m_p$, $U_p(m_p)$, $m_q$ and $U_q(m_q)$ are constructed as in Lemma \ref{hcon}, Proposition \ref{hcon2} and (\ref{rescue}). 

\medskip

\item Let $z_{p_*} = \chi_p(p_*)$ and $z_{q_*} = \chi_q(q_*)$. By Proposition \ref{hcon} and making $\epsilon$ sufficiently small in (2), we can assume that 
$(z_{p_*} , \chi_p(D_p), \chi_p(U_p), J_\infty, g_\infty, h_\infty)$ and $(z_{q_*}, \chi_p(D_q), \chi_q(U_q),  J_\infty,  g_\infty,  h_\infty)$ 
satisfy the $H$-condition after shrinking $U_p$ and $U_q$.
 Now we consider the local section $\sigma_p$ on $\chi_p(U_p)$. Since $X_{reg} = \R$,  we are still able to choose $\sigma_p$ from the construction in Lemma \ref{hcon} appropriately  satisfying the $H$-condition.

\medskip

\item We now apply Proposition \ref{l3} by letting $\tau = \dbar \sigma'_p$. Then there exists $L^{k_{p,q}} $-valued section $u$ solving the $\dbar$-equation $\dbar u = \tau$ with 
$$ \int_{X_\infty} |u|^2_{(h_\infty)^{k_{p,q}}} dV_{g_\infty} \leq \frac{1}{2\pi} \int_{X_\infty} |\tau|^2_{ (h_\infty)^{ k_{p,q} } } dV_{g_\infty} \leq \min(\frac{1}{8\sqrt{2}C}, 10^{-20} ).$$ 
From the $H$-condition, 
$$u(z_p) \leq C ||u||_{L^2(D)} \leq \frac{1}{8}. $$
We set $\sigma''_p = \sigma_p' - u$. Then 
\medskip

\begin{itemize}

\item $ |\sigma_p''(z_{p_*})|_{(h_\infty)^{ k_{p,q}}}^2 \geq 1/2.$

\medskip

\item $||\sigma_p''||_{L^2(X_\infty, k_{p,q} g_\infty, (h_\infty)^{k_{p,q}})} \leq 2(2\pi)^{n/2}.$

\medskip

\item $||\sigma_p''||_{L^2(X_\infty\setminus \chi_p(U_p), k_{p,q} g_\infty, (h_\infty)^{k_{p,q}}) } = ||u||_{L^2(X_\infty, k_{p,q} g_\infty, (h_\infty)^{k_{p,q}}) } \leq \min ((8\sqrt{2}C)^{-1}, 10^{-10}).$ \\

\end{itemize}
Then by Proposition \ref{l22}, 
\begin{eqnarray*}  |\sigma_p''(p)|_{(h_\infty)^{k_{p,q}}} 
&\geq& |\sigma_p''(z_{p_*})|_{(h_\infty)^{k_{p,q}}}  - \sup_{X_\infty} |\nabla \sigma''_p|_{k_{p,q} g_\infty, (h_\infty)^{k_v}} d(p, z_{p_*})  \\
&\geq& 1/4 - 1/100 \geq 2/5.
\end{eqnarray*}
Now we restrict $\sigma''_p$ on $\chi_q(U_q)$.  The condition $H_3$ can still be applied since we have assumed that $k_p=k_q$ and so we have, 
$$|\sigma''_p (z_{q_*}) |^2_{(h_\infty)^{k_{p,q}}} \leq C ||\sigma''_p||_{L^2(X_\infty \setminus \chi_p(U_p), k_{p,q} g_\infty, (h_\infty)^{k_{p,q}}) } \leq 1/8. $$
Similarly by Proposition \ref{l22} and the above arguments, we have 
$$|\sigma''_p (q) |^2_{(h_\infty)^{k_{p,q}}}  \leq 1/10.$$

\medskip

\item The same construction will be applied to obtain $\sigma''_q \in H^0(X, L^{k_{p,q}})$ extending to $X_\infty$ with 
$$ |\sigma''_q (q) |^2_{(h_\infty)^{k_{p,q}}} \geq2/5,  ~|\sigma''_q (p) |^2_{(h_\infty)^{k_{p,q}}} \leq 1/8. $$
Then we can conclude that %
$$\Phi_{k_{p,q}, \sigma} (p) \neq \Phi_{k_{p,q}, \sigma} (q).$$ This completes the proof of the proposition since $\Phi_{k_{p,q}, \sigma}$ is stabilized for sufficiently large $k\in \mathcal{F}(X, L)$.

\end{enumerate}

\end{proof}

\begin{corollary} \label{homeo} $\Phi: X_\infty \rightarrow X$ is a homeomorphism.

\end{corollary}

\begin{proof}  It follows from Proposition \ref{inject},  $\Phi$ is a continuous bijection. Immediately, this implies that   $\Phi$ is a homeomorphism.

\end{proof}

Finally, we are able to prove Theorem \ref{main1}. 
\medskip

\noindent{\it Proof of Theorem \ref{main1} } Combining the results of Proposition \ref{ghlim}, Lemma \ref{extenm} and Corollary \ref{homeo}, we have completed the proof of Theorem \ref{main1}.


\subsection{Applications}

 In this section, we will give applications of Theorem \ref{main1} by  proving Corollary \ref{main3} and Corollary \ref{main4}. 
 
 \medskip
 
 \noindent{\it Proof of Corollary \ref{main3}}
 Suppose $X$ is an $n$-dimensional projective Calabi-Yau manifold. Suppose $L_j \rightarrow X$ is a sequence of $\mathbb{Q}$ line bundles over $X$, $j=1, 2...$ such that $c_1(L_j)$ converges to $c_1(L) \in H^{1,1}(X, \mathbb{Q})$, where $L$ is a big and nef $\mathbb{Q}$-line bundle. Let $g_j$ be the unique Ricci-flat K\"ahler metric in $c_1(L_j)$. Since $L$ is big and nef,  from Proposition \ref{semicr} and Corollary \ref{crepsin}, $L$ is semi-ample and for some sufficiently large $k$, the linear system $|L^k|$ is base point free and induces a unique  surjective birational morphism $\Phi: X \rightarrow Y$ for a normal projective Calabi-Yau variety $Y$. In particular, $Y$ has crepant singularities and $\Phi$ is a crepant resolution. By \cite{EGZ} or more directly by \cite{To1}, there exists a unique Ricci-flat K\"ahler metric $g_Y \in c_1(L)$ on $Y$ with bounded potentials.  Standard argument shows that $g_j$ converges to $g_Y$ on $\Phi^{-1}(Y_{reg})$ smoothly with a uniform $C^0$ bound on their local potentials.  Immediately we can apply Theorem \ref{main1} combined with \cite{RZ} to derive a limiting compact Calabi-Yau metric space $(Y, d_Y)$ as the metric completion of $(Y_{reg}, g_Y)$.  
 
   \qed

 \medskip
 
 \noindent{\it Proof of Corollary \ref{main4}} Let 
 \begin{equation}
\begin{diagram}
\node{X} \arrow{se,b,}{f} \arrow[2]{e,t,..}{ }     \node[2]{X'} \arrow{sw,r}{f'} \\
\node[2]{Y}
\end{diagram}
\end{equation} 
be the flip of two smooth projective Calabi-Yau manifolds $X$ and $X'$. $Y$ is then a normal projective Calabi-Yau variety with crepant singularities. In particular, $f$ and $f'$ are both crepant resolutions of $Y$. Let $L_Y$ be an ample line bundle over $Y$ and $\A$ be an ample $\mathbb{Q}$-line bundle on $X$. Then $L_j = f^*L_Y \otimes j^{-1} \A$ is an ample $\mathbb{Q}$-line bundle over $X$ and there exists a unique smooth Ricci flat K\"ahler metric $g_j\in c_1(L_j)$.   Then we can apply Corollary \ref{main3} and complete the proof of Corollary \ref{main4}.
  
 \qed

We would also like to remark that the gradient estimates in section 3.1 can be very much generalized to a family of degenerate and singular complex Monge-Ampere equations. The upshot is that as long as the equation satisfies certain bound on the Ricci current, the gradient estimate holds with respect to the singular K\"ahler metric induced from the solution instead of a fixed reference metric.


\bigskip

\section{Canonical models of general type with crepant singularities}


 Let $X$ be an $n$-dimensional  canonical model  with crepant singularities. Let $X_{reg}$ and $X_{sing}$ be the regular part and singular part of $X$,   and  
 $$\pi: X' \rightarrow X $$ be a crepant resolution. Then $X'$ is a smooth minimal model of general type, i.e., $K_{X'}$ is big and nef, and so $K_{X'}$ is semi-ample by Kawamata's base point free theorem. $X$ is the unique canonical model of $X'$.  
 
 Let $\chi\in c_1(X)$ be a multiple of the Fubini-Study metric from certain pluricanonical embedding of $X$.   Let $\Omega$ be a smooth volume form on $X$ satisfying $\ddbar\log \Omega = \chi$. Then the following Monge-Ampere equation can be solved \cite{EGZ, Z} for a unique $\varphi_{KE} \in L^\infty(X)\cap PSH(X, \chi)\cap C^\infty(X_{reg})$
\begin{equation}\label{keq1}
(\chi+ \ddbar \varphi)^n = e^{\varphi} \Omega.
\end{equation}
Without confusion, we also identify equation (\ref{keq1}) as one on $X'$ after pullback. This implies that the Ricci current of $\omega_{KE}  = \chi+ \ddbar \varphi_{KE}$ is well-defined and satisfies the K\"ahler-Einstein equation 
\begin{equation}\label{keq2}
Ric(\omega_{KE}) = -\omega_{KE}
\end{equation}
smoothly on $X_{reg}$ and globally on $X$ in the sense of distribution. Let $g_{KE}$ be the associated K\"ahler-Einstein metric on $X_{reg}$ and $h_{KE} = ( (\omega_{KE})^n)^{-1}$ be the hermitian metric on $K_X$. 

Note that $\log (\omega_{KE}) ^n$ is locally plurisubharmonic on $X$.  
In fact, equation (\ref{keq2}) admit a unique solution in $c_1(X)$ with bounded local potentials. Let $g_{KE}$ be the K\"ahler metric associated to $\omega_{KE}$. 

\subsection{A graident estimate} In this section, we will prove the following gradient estimate. 

\begin{proposition} \label{gradeke} There exists $C>0$ such that if $\varphi_{KE} \in L^\infty(X)\cap PSH(X, \chi) \cap C^\infty(X_{reg})$ solves equation (\ref{keq1}), then, 
\begin{equation}\label{gradkeq}
\sup_{X_{reg}} |\nabla \varphi_{KE}|_{g_{KE}}^2 \leq C.
\end{equation}
\end{proposition}

\begin{proof} We consider the following normalized Kahelr-Ricci flow on the minimal model $X'$ starting with any smooth initial K\"ahler metric $g_0$ whose corresponding K\"ahler form is $\omega_0$, 
\begin{equation}\label{krf}
\ddt{g} = - Ric(g) - g. 
\end{equation}
The flow  is then equivalent to the following parabolic Monge-Ampere equation
\begin{equation}\label{krf3}
\ddt{\varphi} = \log \frac{ ( (1-e^{-t}) \chi' + e^t \omega_0 + \ddbar \varphi )^n}{\Omega} + \varphi, ~\varphi|_{t=0} = 0, 
\end{equation}
where $\chi'= \pi^*\chi$. It is shown in \cite{Z2, ST4} that there exists $C>0$ such that for all $t\geq 0$, 
$$\sup_{X'}  \left| \nabla \left( \ddt{\varphi(t) } + \varphi(t) \right) \right|^2_{g(t)} \leq C. $$

Let $E= \pi^{-1} (X_{sing})$. It is well-known that  $\varphi(t)$ converges to the limit solution $\varphi_{KE}$ solving  equation (\ref{keq1}) in $L^\infty(X') \cap PSH(X', \chi')\cap C^\infty(X'\setminus E)$ and $\ddt{\varphi(t)}$ converges to $0$ in $C^\infty(X'\setminus E)$. Then the lemma is proved by letting $t\rightarrow \infty$.

\end{proof}

\subsection{Limiting metric spaces}

We now construct a family of almost K\"ahler-Einstein metrics   \cite{TW} on $X'$.  We fix a smooth K\"ahler metric $g_{\A}$ in a fixed K\"ahler class $\A$ on $X'$ whose associated K\"ahler metric is $g_\A$. Then we consider the following equation for $g_t $, 
\begin{equation}\label{alke}
Ric(g_t) = - g_t + e^{-t} g_\A, ~ t\in [0, \infty).
\end{equation}
Obviously, $-c_1(X) + e^{-t} \A $ is a K\"ahler class and equation (\ref{alke}) is equivalent to the following Monge-Ampere equation 
\begin{equation}\label{alke2}
(\chi' + e^{-t} \omega_\A + \ddbar \varphi_t )^n = e^{-\varphi_t } \Omega,
\end{equation}
where
$\chi'=\pi^*\chi$, $\omega_t = \chi' + e^{-t} \omega_\A + \ddbar \varphi_t$ be the K\"ahler form associated to the K\"ahler metric $g(t)$ solving equation (\ref{alke}).

\begin{lemma} \label{lem41}There exists a unique smooth K\"ahler metric $g_t \in -c_1(X) + e^{-t} \A $ solving equation \ref{alke} for all $t>0$. There exists $C>0$ such that 
$$\sup_{X'} |\varphi_t |\leq C, ~~ tr_{g_t}(\chi') \leq C $$ for all $t\in [0, \infty)$. Furthermore, $\varphi_t$ converges to $\varphi_{KE}$, the unique solution of equation (\ref{keq1}),  in $C^\infty(X'\setminus E)$.

\end{lemma}

We will verify in the following lemma the almost K\"ahler-Einstein condition introduced in \cite{TW}. 

\begin{lemma} \label{locesk} Let $g_j $ be the solution of equation (\ref{alke}) for $t=j$. Then $g_j$ satisfies the following almost K\"ahler-Einstein conditions.

\begin{enumerate}

\item $Ric(g_j) \geq -g_j$, 
\smallskip

\item there exists $p\in X'\setminus E$ and $r_0, \kappa >0$ such that for all $j=1, 2, ...$, 
$$B_{g_j}(p, r_0)\subset\subset X'\setminus E, ~Vol(B_{g_j}(p, r_0) ) \geq \kappa,$$ 

\item  Let $g_j(t)$ be the solution of the normalized K\"ahler-Ricci flow 
$$\ddt{g_j(t)} = -Ric(g_j(t)) - g_j(t), ~ g_j(0) = g_j. $$ Then 
$$ \lim_{j\rightarrow \infty} \int_0^1 \int_{X'} |R(g_j(t)) + n| dV_{g_j(t)} dt =0. $$

\end{enumerate}

\end{lemma}

\begin{proof} (1) and (2) follow easily from Lemma \ref{lem41}. Notice that the minimum of the scalar curvature is non decreasing along the Ricci flow while $R(g_j(0)) > -n$. Therefore
\begin{eqnarray*}
\int_0^1\int_{X'} |R(g_j(t))+n|dV_{g_j(t)}dt &\leq & \int_{X'} \int_0^1( R(g_j(t))+n) dV_{g_j(t)} dt \\
&=& \int_0^1e^{-j-t}\A \cdot ( e^{-j-t} \A + c_1(X) )^{n-1}dt \rightarrow 0
\end{eqnarray*}
as $j\rightarrow \infty$.

\end{proof}

We then apply the main results of Tian-Wang \cite{TW} to obtain the following proposition. 

\begin{proposition} \label{ghlimgt} Let $(X', p, g_j)$ be the almost K\"ahler-Einstein manifolds in the assumption of Lemma \ref{locesk}. Then $(X', p, g_j)$ converges to  a metric length space $(X_\infty, p_\infty, d_\infty)$ satisfying 

\begin{enumerate}

\item $\R$, the regular set of $X_\infty$, is a smooth open dense convex set in $X_\infty$, 

\smallskip

\item the limiting metric $d_\infty$ induces a  smooth K\"ahler-Einstein metric $g_{KE}$ on $\R$ satisfying $Ric(g_{KE}  ) = - g_{KE} $,  

\smallskip

\item the singular set $\mathcal{S}$ has Hausdorff dimension no greater than $2n-4$.

\end{enumerate}

\end{proposition}

We would like to show that $\mathcal{R} = X_{reg}$ following the ideas in \cite{RZ}. We first need the following lemma due to Rong-Zhang (Theorem 5.2 \cite{RZ}). It is possible that the elliptic arguments in \cite{CDS3} for the Gromov-Hausdorff limit of conical K\"ahler-Einstein metrics can also be applied to derive Proposition \ref{ghlimgt}.

\begin{lemma} \label{RZ1} Let $(M, g)$ be an $n$-dimensional projective manifold with a K\"ahler metric $g$. Suppose that $E$ is a  subvariety of dimension $0<m<n$ and $p$ is point in the regular part of $E$. Then if
$$ \sup_{B_g(p, \frac{2\pi}{\Lambda})} |Sec(g)| \leq \Lambda, $$
for some $p\in M$ and $\Lambda>0$, 
then 
$$Vol_g(B_g(p, r)\cap E) \geq \delta(m, \Lambda) r^{2m},$$
for any $r\leq min\{ i_g(p), \frac{\pi}{2\Lambda}\}$, where $i_g(p)$ is the injectivity radius of $g$ at $p$ and $\delta=\delta(m, \Lambda)>0$ only depends on $m$ and $\Lambda$.

\end{lemma}

We now define $S_E$ by 
$$S_E=\{ p\in X_\infty~|~ there~exist~ \{p_j \}_{j=1}^\infty \subset E ~with ~ p_j \rightarrow p, ~as~j\rightarrow \infty \}.$$  Obviously, $S_E$ is closed.

\begin{lemma} \label{le45} $S_E \subset \mathcal{S}. $ \end{lemma}

\begin{proof} We will have to combine the argument in \cite{RZ}  and \cite{TW} since the Ricci curvature of $g_j$ is not uniformly bounded. We prove by contradiction. Suppose there exists $p\in S_E \cap \mathcal{R}$ with $p_j \in E\subset (X, g_j)$ converging to $p \in (X_\infty, d_\infty)$. We can always assume that $p_j$ lies in the regular part of an $m$-dimensional component $E'$ of $E$ by arbitrarily small perturbation after passing to a subsequence, for some $m>0$. Then for any sufficiently small $\delta>0$, there exists $r_0>0$ sufficiently small such that  $$Vol(B_{d_\infty}(p, r_0)) > (1-\delta) Vol(B_{Eucl}(0, r_0)),$$ where $B_{Eucl}(0, r_0)$ is the Eucidian ball centered at $0$ with radius $r_0$. Without loss of generality, we can also assume that $Vol(B_{g_j}(p_j, r_0)) > (1-\delta) Vol(B_{Eucl}(0, r_0))$ for sufficiently large $j$. One then consider the solution $g_j(t)$ of the normalized K\"ahler-Ricci flow starting with $g_j$ for $t\in [0, 1]$. By \cite{TW}, $B_{g_j(t)}(p_j, r_0)$ and $B_{d_\infty}(p, r_0)$ can be arbitrarily small for sufficiently large $j$ and all $t\in [0, 1]$. The modified version of Perelman's psedoulocality theorem  (Proposition 3.1 \cite{TW}) can be applied and after choosing a sufficiently small uniform constant $\delta$, there exists $K, r >0$ such that 
$$|Rm_{g_j(t)} (y) |_{g_j(t)} \leq K t^{-1}, ~y \in B_{g_j}(p_j, r_0/2) $$
for sufficiently large $j$. By choosing $t=1/2$, one immediately shows that the curvature of $g_j(1/2)$ is uniformly bounded by $2K$ on $B_{g_j(1) } (p_j, r_0/4) $ for sufficiently large $j$. 

By Lemma \ref{RZ1}, there exists $\varepsilon>0$, such that for all sufficiently large $j$,  
$$Vol_{g_j(1/2)} (B_{g_j(1/2)}(p_j, r_0/4)\cap E) \geq \varepsilon. $$
However, letting $\omega_j(t)$ be the K\"ahler form associated to $g_j(t)$, we have 
$$Vol_{g_j(1/2)} (B_{g_j(1/2)}(p_j, r_0/4)\leq \int_{E'} (\omega_j(1/2))^m = E' \cdot [\omega_j(1/2)]^m \rightarrow 0= e^{-j} ( E'\cdot~ \A^m) \rightarrow 0,$$
as $ j\rightarrow \infty$ because $K_{X'}$ is numerically zero on $E$. Contradiction.

\end{proof}

The following lemma is the pointed version of Theorem 4.1 in \cite{RZ}

\begin{lemma} \label{RZ2} There exists a homeomorphic local isometry $f: (X'\setminus E,d_{g_\infty}) \rightarrow (X_\infty \setminus S_E, d_\infty)$, where $d_{g_\infty}$ is the metric induced by $g_\infty$ on $X'\setminus E$ (or $X_{reg}$).

\end{lemma}

\begin{proposition} $ \mathcal{R} = X_{reg} $.

\end{proposition}

\begin{proof} From Lemma \ref{RZ2}, all the points in $X_\infty\setminus S_E$ must lie in $\mathcal{R}$, since $(X_{reg}, d_{g_\infty})$ is smooth. Hence $\mathcal{S} \subset S_E$. The proposition immediately follows by Lemma \ref{le45}.

\end{proof}

\subsection{Local $L^2$-estimates}

 We now pick a base point $p$ in $X_{reg}$ as in Lemma \ref{locesk}. As shown in section 4.2, $(X', p, g_j)$ converges to the metric length space of $(X_\infty, p_\infty, d_\infty)$. The difficulty here is that we do not have a uniform bound on the diameter of $(X', g_j)$, unlike the case of Calabi-Yau varieties. In order to obtain a uniform diameter bound, we will prove by contradiction,  and assume that 
$$diam_{g_j}(X') \rightarrow \infty.$$
This immediately implies that $(X_\infty, d_\infty)$ is not a compact metric space as its diameter is infinite.  

Let $B_\infty(r)$ be the geodesic ball in $(X_\infty, d_\infty)$ centered at $p_\infty $ with radius $r>0$. Let $B_j(r)$ be the geodesic ball of $(X', g_j)$ centered at $p$ of radius $r$. Then $B_j(r)$ converges to $B_\infty(r)$ in Gromov-Hausdorff topology as $j\rightarrow \infty$.  We will derive local $L^2$-estimates on each $B_\infty(r)$ for all $r>0$.

\begin{lemma} \label{soho} Suppose $diam_{g_j}(X') \rightarrow \infty$ as $j\rightarrow \infty$. Fix any $0< r <R$, the Sobolev constant on $B_j(r)$ is uniformly bounded below by a constant $C_S$ depending on upper bound of $R$, $R^{-1}$ and $(R-r)^{-1}$. More precisely, for any $j=1, 2, ..., $ and any $f\in L^{1,2}(X', g_j)$ with support in $B_j(r)$, 
\begin{equation}
 ||\nabla f|_{L^2(B_j(r), g_j)} \leq C_S  ||f||_{L^{\frac{2n}{n-1}}(B_j(r), g_j)}.
 \end{equation}

\end{lemma}

\begin{proof} First, we consider $B_j(r)$ and $B_j(R)$ for $0<r<R$. Then  the Dirichlet  isoperimetric constant $C_I$ is bounded by the following estimate (c.f. Corollary 10.2 in \cite{Li})
\begin{equation}
C_I (B_j(r)) \geq C \left(\frac{V(B_j(R) ) - V(B_j(r))}{V_{-1} (r+R)} \right)^{\frac{2n+1}{2n} }
\end{equation}
as the Ricci curvature of $g_j$ is uniformly bounded below by $-1$, where $V_{-1}(r+R)$ is the volume of a geodesic ball of radius $(r+R)$ in the connected space form of constant $-1$ sectional curvature, and $C$ is a constant only dependent on $n$ (c.f. Corollary 10.2 in \cite{Li}). 
Since $B_j(R)$ and $B_j(r)$ converge in Gromov-Hausdorff topology and measure to $B_\infty(R)$ and $B_\infty(r)$, $C_I(B_j(r))$ is uniformly bounded below by a constant dependent on the upper bounded of $R$, $R^{-1}$ and $(R-r)^{-1}$. Then the lemma follows by the relation between the Sobolev constant $C_S$ and  the isoperimetric constant $C_I$.

\end{proof}

We can now derive the following local $L^2$-estimates. 

\begin{proposition} \label{ll31}   For any $R>0$, there exists  $K_R >0$ such that if  $s\in H^0(X, (K_X) ^k)$ for $k\geq 1$, then 
\begin{equation}\label{lll1}
\|s\|_{L^{\infty, \sharp}(B_\infty(R))} \leq K_R \|s\|_{L^{2, \sharp}(B_\infty(2R) )}
\end{equation}
\begin{equation}\label{lll2}
 \|\nabla s\|_{L^{\infty, \sharp}(B_\infty(R)) } \leq K_R \|s\|_{L^{2, \sharp}(B_\infty(2R))} , 
\end{equation}
where the norms are taken with respect to $(h_{KE})^k$ and $k\omega_{KE}$.

\end{proposition}

\begin{proof} We break the proof into the following steps.

\begin{enumerate}

\item From the same calculations, we have on $X_{reg}$ 
$$\Delta |s| \geq - |s|, $$ where $\Delta$, $\nabla$ are covariant derivatives with respect to $kg_{KE}$ and $h_{KE}=(\omega_{KE})^{-n}$. For any $1<r<R$, we choose a smooth cut-off function $F_{r, R} $ on $\mathbb{R}^+$ such that $0\leq F_{r, R} \leq 1$, $F_{r, R} (x)=1$ for $x\in [0, r]$ and $F_{r, R} (x)=0$ on $[R, \infty)$. We can assume that $|F'| \leq A(R-r)^{-1}$ for some uniform constant $A>0$. Then we let $\eta  (z) = F_{r, R}(\rho/k^{1/2})$, where $\rho$ is the geodesic distance from $p$ to $z\in X_\infty$. Then straightforward calculations show that 
$$
\int_{X'} \left|\nabla  \left(\eta  |s|^{\frac{p+1}{2}} \right) \right|^2 dV_{kg_{KE}} \leq Cp \int_{X'} (\eta^2+ |\nabla \eta|^2) |s|^{p+1} dV_{kg_{KE}}
$$
and so,
\begin{equation}
\int_{B_\infty(r)} |\nabla  \left(   |s|^{\frac{p+1}{2}} \right)|^2 \leq \frac{Cp}{(R-r)^2} \int_{B_\infty(R)}  |s|^{p+1}
\end{equation}
for some uniform constant $C$ because $|\nabla \eta|$ is bounded by $A(R-r)^{-1}$.

\item We would like to prove a Sobolev type inequality for $g_{KE}$.

\begin{claim} For any $r>0$, there exists $K>0$ such that for all $f \in L^\infty (B_\infty(r)) \cap L^{1, 2}(B_\infty(r), k\omega_{KE})$ with compact support in $B_\infty(r)$,   
\begin{equation}
|| \nabla f ||_{L^2(B_\infty(r), kg_{KE} )}  \geq K || f ||_{L^{\frac{2n}{n-1}}(B_\infty(r), k g_{KE})}  .
\end{equation}

\end{claim}

\begin{proof}  We first prove that the Sobolev inequality holds for $f_\epsilon = \rho_\epsilon f$ with $\rho_\epsilon$ constructed from Lemma \ref{app} for $Y=X$ and $Z=X_{sing}$. Let $\Omega_\epsilon = Supp f_\epsilon$. Then $(\Omega_\epsilon, g_j)$ converges to $(\Omega_\epsilon, g_{KE})$ smoothly as $j\rightarrow \infty$. Therefore $\Omega_\epsilon \subset ( B_j(r), g_j)$ for sufficiently large $j$. Therefore from Lemma \ref{soho}, we have 
$$|| \nabla_{kg_j} f_\epsilon ||_{L^2(B_j(r), kg_j)} \geq C_S ||f_\epsilon||_{L^{\frac{2n}{n-1}}(B_j(r), k g_j)} $$
since the Sobolev inequality is scaling invariant. 
By letting $j \rightarrow \infty$, we have
$$|| \nabla  f_\epsilon ||_{L^2(B_\infty (r), kg_{KE})} \geq C_S ||f_\epsilon||_{L^{\frac{2n}{n-1}}(B_\infty (r), kg_{KE})}  $$
since $g_j$ converges smoothly to $g_{KE}$ smoothly on $\Omega_\epsilon$.
Then the claim can be proved by the same argument as in Proposition \ref{l21} after letting $\epsilon \rightarrow 0$.

\end{proof}

\item Combining the above estimates, we have 
$$
C_S \left|\left|  |s|^{\frac{p+1}{2}} \right|\right|^2_{L^{\frac{2n}{n-1}}(B_\infty(r), kg_{KE})}   \leq \frac{Cp}{(R-r)^2} \int_{B_\infty(R)}   |s|^{p+1} dV_{kg_{KE}}
$$
and so for some uniform constant $C>0$, we have 
$$ || |s| ||_{L^{\frac{n}{n-1} (p+1)}(B_\infty(r), k g_{KE}) }  \leq \left( \frac{Cp}{(R-r)^2} \right)^{\frac{1}{p+1}} ||s||_{L^{p+1}(B_\infty(R), kg_{KE})}. $$
Let $\beta = n(n-1)^{-1}$. We then have  
$$|| s ||_{L^{2\beta^{k+1} }(B_\infty(r_{k+1}) , k g_{KE} )}  \leq  \left( \frac{C\beta^k}{ (r_k - r_{k+1})^2 } \right)^{\frac{1}{\beta^{k+1}} } ||s||_{L^{2\beta^k}(B_\infty(r_{k}), k g_{KE})}, $$
for $k=1, 2, ...$, and $r_k$ is an decreasing sequence to be determined. 
By letting $r_k = R - r + \frac{r}{2^k B}$ for sufficiently large fixed $B>0$. One can apply the standard Nash-Moser iteration and we obtain that 
$$||s||_{L^\infty(B_\infty(r) )} \leq K ||s||_{L^2(B_\infty(R), k g_{KE}) }. $$

\item By apply the arguments above and those in the proof of Proposition \ref{l22}, one can also prove (\ref{lll2}). 

\end{enumerate}

\end{proof}

\begin{proposition} \label{lll3} Let $X$ be a normal projective variety with crepant singularities and ample canonical divisor $K_X$. Let $\omega \in - kc_1(X)$ be the unique K\"ahler-Einstein current with bounded local potentials and $h$ be a hermitian metric on $(K_X)^k $ satisfying 
$$Ric(\omega) = - \frac{1}{k} \omega, ~ \omega=Ric(h) $$ for any $k\geq 2$. then for any smooth $(K_X)^k$-valued $(0,1)$-form $\tau$ satisfying

\begin{enumerate}

\item $\dbar \tau =0$,  

\item $Supp ~\tau \subset \subset X_{reg}$, 

\end{enumerate}
there exists a $(K_X)^k$-valued section $u$ such that $\dbar u = \tau$ and $$ \int_X |u|^2_h \omega^n \leq \frac{1}{\pi} \int_X |\tau|^2_{h, \omega} \omega^n. $$

\end{proposition}

\begin{proof} The proof is very similar to that of Proposition \ref{l3} by approximation. Let $\pi: X'\rightarrow X$ be the crepant resolution. Since $K_{X'}$ is big and semi-ample, there exists a divisor $D$ such that $K_{X'}- \epsilon_0 D$ is ample for some sufficiently small $\epsilon_0 >0$.  We let $h_\Omega = \Omega^{-1}$ be the smooth hermitian metric on $K_{X'}$ satisfying $Ric(h_\Omega) = \chi$ and a a smooth hermitian metric $h_D$ of $[D]$ such that  
$$ \chi - \epsilon Ric(h_D) \in [K_X] - \epsilon [D] $$ is a  K\"ahler metric on $X'$ for sufficiently small $\epsilon>0$. Then we consider the Monge-Ampere equation
$$(\chi -\epsilon^2 Ric(h_D) + \ddbar \varphi_\epsilon) ^n = e^{(1+\epsilon)\varphi_\epsilon} \Omega.$$
Let 
$$h_\epsilon = e^{- (\varphi_\epsilon + \epsilon^2 \log |s_D|^2_{h_D}) } h_\Omega, ~\omega_\epsilon = \chi - \epsilon^2 Ric(h_D) + \ddbar \varphi_\epsilon, ~ \alpha_\epsilon =\chi - \epsilon^2(1+\epsilon) Ric(D). $$
Then 
$$Ric(h_\epsilon) = \omega_\epsilon, ~Ric(\omega_\epsilon) = -(1+\epsilon) \omega_\epsilon + \alpha_\epsilon \geq -(1+\epsilon)\omega_\epsilon. $$ Then we can apply  Theorem \ref{demailly} by writing $kK_{X'} = (k-1)K_{X'} + K_{X'}$ because  the Ricci curvature of $\omega_\epsilon$ is bounded below and converges to constant.  One then can proceed as in the proof of Proposition \ref{l3}. In particular, we make use of the fact that, after talking $\epsilon \rightarrow 0$, the limiting hermitian metric $h = (\omega ^n)^{-k}$.

\end{proof}

\subsection{Local separation of points}

We  will assume that $diam_{g_j}(X') = \infty$ and derive contradiction. Due to the gradient estimates (\ref{gradkeq}) in Proposition \ref{gradeke}, one can derive the same conclusion as in Lemma \ref{l200} by considering  the balls $B_\infty(r)$ with $r\rightarrow \infty$. Hence we have the following lemma similar to Lemma \ref{extenm} and Corollary \ref{24} because each holomorphic section in $H^0(X, (K_X)^k)$ is $L^2$ integrable on $X_{reg}$ with respect to $(h_{KE})^k$ and $\omega_{KE}$.

\begin{lemma}\label{extenmm}
Any holomorphic section $\sigma\in H^0(X, K_X^k)$ can be continuously extended from $X_{reg}$ to $X_\infty$. Furthermore, 
The map $$\Phi_{k, \sigma} : X _\infty \rightarrow X \subset \mathbb{CP}^{d_k}$$ is a Lipschitz map for sufficiently large $k\in \mathcal{F}(X, K_X)$ with respect to  $h_{KE} = (\omega_{KE}^n)^{-1}$ on $(X_\infty, d_\infty)$, where $\sigma=\{ \sigma_1, ..., \sigma_{d_k+1}\}$ is a basis of $H^0(X, (K_X)^k)$.

\end{lemma}
$\Phi_{k, \sigma}$ will be stabilized for sufficiently large $k\in \mathcal{F}(X, K_X)$ and $\Phi|_{(X_\infty)_{reg}}=id$. Without loss of generality, we write $\Phi= \Phi_{k, \sigma}$. Without the diameter bound, $(X_\infty, d_\infty)$ might not be a a compact metric length space and $\Phi$ might not be surjective. However, we can show that $\Phi$ is in fact injective.

\begin{proposition} \label{isoset} The map $\Phi: X_\infty \rightarrow X$ is injective.  \end{proposition}

\begin{proof}   We fix a base point $p_\infty\in X_\infty $ once for all and let $B_\infty(r)$ be the geodesic ball in $(X_\infty, d_\infty)$ centered at $p_\infty$ with radius $r$. For any two distinct points $p, q \in B_\infty(R)$, with Proposition \ref{ll31} and Proposition \ref{lll3}, we are able to apply the $H$-condition and imitate the arguments in section 3.5 to find two section $\sigma_p$ and $\sigma_q$ in $H^0(X, (K_X)^{k_{p,q}})$ for some sufficiently large $k_{p, q}\in \mathcal{F}(X, K_X)$ such that 
$$|\sigma_p(p)| \geq 2/5, ~~|\sigma_p(q)|\leq 1/10, ~  |\sigma_q(q)| \geq 2/5, ~~|\sigma_q(p)|\leq 1/10. $$
Therefore $$\Phi(p)\neq \Phi(q) $$
and so $\Phi$ is injective.

\end{proof}

\begin{lemma} \label{46} Suppose $diam_{d_\infty}(X_\infty) = \infty$. Then 

\begin{enumerate}
\item

$X\setminus \Phi(X_\infty)\neq \phi$, 

\item  for any sequence of points $q_j \in X_\infty$   with $\lim_{j\rightarrow \infty} d_\chi(\Phi(q_j), P) = 0$  for some point   $P\in X\setminus \Phi(X_\infty)$,
$$d_\infty(p_\infty, q_j) \rightarrow \infty$$
as $j \rightarrow \infty$.

\end{enumerate}

\end{lemma}

\begin{proof} We prove the lemma by contradiction. 

\begin{enumerate} 

\item Suppose $X=\Phi(X_\infty)$. There exists a sequence of points $Q_j\in X_{reg}$ with $d_\infty(p_\infty, Q_j) \rightarrow \infty$ as $j\rightarrow \infty$. Without loss of generality, we can assume there exists $Q'\in X_\infty$ such that $d_\chi(\Phi(Q_j), \Phi(Q')) \rightarrow 0$, after passing to a sequence. Then an open ball centered a $Q'$ must contain all $Q_j$ for sufficiently large $j$. Contradiction.

\item Suppose not.  Then there exist a sequence of points $q_j \in X_\infty$ such that 
$$d_\infty(p_\infty, q_j) \leq A, ~ d_\chi(\Phi(q_j), P) \rightarrow 0$$ 
for some fixed constant $A>0$ and $P\in X\setminus \Phi(X_\infty)$. Since $X_{reg}$ is open and dense in $(X_\infty, d_\infty)$ and $d_\infty$ is bounded below by a multiple of $\chi$, there exist a sequence of points $q_j'\in X_{reg}$ with
$$d_\infty(p_\infty, q'_j) \leq A, ~ d_{\chi} (\Phi(q_j'), P) \rightarrow 0. $$ By compactness, $q'_j$ converges to $q'_\infty \in X_\infty$ with respect to $d_\infty$. Then 
$$ d_\infty(q'_\infty, q'_j) \rightarrow 0 $$
 and so 
 $$ d_\chi( \Phi(q'_\infty), P) = \lim_{j \rightarrow \infty} d_\chi ( \Phi (q'_\infty), \Phi(q'_j)) = 0 .$$
Hence $P=\Phi(q'_\infty) \in  \Phi(X_\infty)$ and contradiction. 

\end{enumerate}

\end{proof}

\begin{corollary}  \label{co41} $\Phi(X_\infty) $ is open dense in $X$. In particular,  $\mathcal{S}  \subset  X_{sing}$, and $X_{sing} \setminus \Phi(\mathcal{S} )$ is a closed set in $X$ with respect to the Fubini-Study metric.

\end{corollary}

\begin{proof}  
$\Phi(X_\infty)$ is obviously dense in $(X, \chi)$ as $\Phi(\R) = X_{reg}$.  We now show that $\Phi(X_\infty)$ is open. Suppose not, then there exists $q\in X_\infty$ such that   any geodesic open ball in $(X, \chi)$ centered at $\Phi(q)$ is not entirely in $\Phi(X_\infty)$. Then one can pick a sequence of points $y_j\in X$ satisfying 
$$d_\chi( y_j, \Phi(q) )\rightarrow 0, ~~ y_j \notin \Phi(X_\infty). $$ We can pick $x_j \in X_\infty $ such that 
$$d_\chi (\Phi(x_j), y_j) \leq j^{-1}, ~d_\infty(x_j, p_\infty) \geq j $$
because of Lemma \ref{46}. We can further assume $x_j \in \R $ since $\R$ is dense in $X_\infty$.
Therefore, 
$$ d_\chi( \Phi(x_j), \Phi(q)) \rightarrow 0, ~  d_\infty( x_j, p_\infty)\rightarrow \infty . $$
   Since $q$ is the limit of a sequence of points $\Phi(x'_j) = x'_j$ in $(X_\infty)_{reg}$ with 
   $$d_\infty(x'_j, q) \rightarrow 0, ~d_{\chi}(\Phi(x'_j), \Phi(q))\rightarrow 0,  $$  we can then join $x_j$ and $x'_j$ by a continuous path of regular points in the geodesic ball in $(X, \chi)$ centered at $\Phi(q)$ with radius equal to $\max(d_\chi(\Phi(q), \Phi(x_j) ), d_\chi(\Phi(q), \Phi(x'_j)))$. We can then by continuity, choose a sequence of points satisfying 
   $$z_j \in \R, ~1\leq d_\infty(q,  z_j ) \leq 10, ~d_\chi( \Phi(z_j), \Phi(q)) \rightarrow 0 . $$ Then by compactness and after passing to a subsequence, we can assume $z_j$ converges to $z_\infty\in (X_\infty, d_\infty)$ with $$1\leq d_\infty(z_\infty, q) \leq 10  . $$ Obviously, $$\Phi(z_\infty) = \Phi(q) $$ and it contradicts $\Phi$ being injective. 

Then immediately, $X_{sing}\setminus (\Phi(\mathcal{S})) = X\setminus \Phi(X_\infty)$ must be closed.

\end{proof}

Applying the same argument in Corollary \ref{co41}, we have the following corollary. 

\begin{corollary} \label{co42} $\Phi:X_\infty$ is a homeomorphism from $(X_\infty, d_\infty)$ to $(\Phi(X_\infty), \chi)$.

\end{corollary}

\subsection{Diameter bound}

In this section, we will prove a diameter bound for $(X_\infty, d_\infty)$ using a trick developed by the author and Weinkove \cite{SW2}. We first we consider the a log resolution 
$$\pi_1: Z \rightarrow X$$
 such that $(\pi_1)^{-1}(X_{sing})$ is the support of  a combiniation of finitely many smooth divisors with simple normal crossings on $Z$. We pick a point $O$  on the smooth part of the divisor and blow up $Z$ at $O$ with 
 $$\pi_2: \tilde X \rightarrow Z.$$ Let 
 $$\tilde \pi = \pi_1\circ\pi_2: \tilde X \rightarrow X. $$
Then we have the following adjunction formula because $X$ has at worst canonical singularities
$$K_{\tilde X} = \tilde\pi^* K_X + (n-1) E  + D, ~D=\sum_{j=1}^m a_j [D_j], $$
where $(n-1)E+ D$ is the exceptional divisor of $\tilde \pi$, $D_j$ are effective prime smooth divisors on $\tilde X$ with $ a_j \geq 0$ for $j=1, ..., m$, $E$ is the exceptional divisor of $\pi_2$ isomorphic to $\mathbb{CP}^{n-1}$.

Since $\tilde \pi^* K_X $ is big and semi-ample, by Kodaira's lemma, there exists an effective divisor $D'$ such that its support coincides with the support of the exceptional divisors of $\tilde\pi$ and 
$$\tilde \pi ^*K_X - \epsilon D'$$ is ample for all sufficiently small $\epsilon>0$. Let $\sigma_E$, $\sigma_D$ and $\sigma_{D'}$  be the defining sections of $E$, $D$ and  $D'$. Here we consider $\sigma_E$, $\sigma_D$ and $\sigma_{D'}$ be the multivalued holomorphic sections which become  global holomorphic sections after taking some power. Let $h_E, h_D, h_{D'} $ be smooth hermitian metrics on the line bundles associated to $E$, $D$ and $D'$ such that 
$$ (\tilde\pi)^* \Omega = |\sigma_E|^{2(n-1)}_{h_E} |\sigma_D|^2_{h_D} \tilde \Omega, ~\tilde \chi - \epsilon Ric(h_{D'}) >0 $$ for a smooth volume form $\tilde \Omega$ on $\tilde X$ and for all sufficiently small $\epsilon>0$, where $\tilde \chi = ( \tilde \pi)^*\chi$.

Let $\tilde \omega$ be a fixed smooth K\"ahler form on $\tilde X$. Then the K\"ahler-Einstein equation lifted to $\tilde X$ is equivalent to the following degenerate Monge-Ampere equation
$$(\tilde\chi + \ddbar \tilde \varphi_{KE})^n = e^{\tilde \varphi_{KE} } (\tilde \pi)^*  \Omega, $$
where $\tilde\varphi_{KE} = (\tilde\pi)^*\varphi_{KE}.$
We consider the following family of Monge-Ampere equations
\begin{equation}\label{4365}
(\tilde\chi + \epsilon \tilde\omega + \ddbar \tilde \varphi_\epsilon )^n = e^{\tilde\varphi_\epsilon }  (|\sigma_E|^{2(n-1)}_{h_E} +\epsilon) (|\sigma_D|^2_{h_D}+\epsilon) \tilde \Omega
\end{equation}

Let $\tilde\omega_\epsilon = \tilde \chi + \epsilon \tilde\omega + \ddbar \tilde \varphi_\epsilon$. By Yau's theorem, equation (\ref{4365}) always admits a unique smooth solution $\tilde\varphi_\epsilon$ for all sufficiently small $\epsilon>0$. 

\begin{lemma} \label{riup} There exists $A>0$  such that if  $\tilde\varphi_\epsilon$ solves equation (\ref{4365}) for some $\epsilon \in (0, 1)$, we have on $\tilde X$, 
$$Ric(\tilde\omega_\epsilon) \leq -\tilde\omega_\epsilon + A \tilde \omega. $$

\end{lemma}

\begin{proof} First we notice that for any positive smooth function $f$ on $\tilde X$
$$ \ddbar \log (f+\epsilon) \geq \frac{f}{f+\epsilon} \ddbar \log f $$
on $\{ f>0\}$ by straightforward calculations.
Therefore, on $\tilde X \setminus (Supp E\cup Supp D)$, we have 
$$\ddbar \log (|\sigma_E|_{h_E}^{2(n-1)}+ \epsilon) \geq  - \frac{(n-1)|\sigma_E|^{2(n-1)}} {|\sigma_E|^{2(n-1)}+\epsilon}  Ric(h_E) \geq -A\tilde \omega, $$
$$\ddbar \log (|\sigma_D|_{h_D}^{2(n-1)}+ \epsilon) \geq  - \frac{ |\sigma_D|^{2(n-1)}} {|\sigma_D|^{2(n-1)}+\epsilon}  Ric(h_D) \geq -A\tilde \omega $$
for some fixed sufficiently large $A>0$. Therefore, 
$$Ric(\tilde \omega_\epsilon) \leq - \tilde \omega_\epsilon + (2A-\epsilon) \tilde \omega +(n-1) Ric(h_E)+ Ric(h_D) $$
on $\tilde X \setminus (Supp E\cup Supp D)$ and naturally extends to $\tilde X$. Then the lemma immediately follows.

\end{proof}

\begin{lemma}  \label{483} Let  $\tilde\varphi_\epsilon$ be the smooth solution for the  equation (\ref{4365}) for  $\epsilon\in (0, 1)$. Then there exist $\lambda$, $C>0$ such that for all $\epsilon\in (0,1)$, we have on $\tilde X$, 
\begin{equation}
\sup_{\tilde X} |\tilde \varphi_\epsilon| \leq C, ~\tilde \omega_\epsilon \leq C   |\sigma_{D'}|_{h_{D'}}^{2\lambda} \tilde \omega.
\end{equation}
Furthermore, $\tilde \varphi_\epsilon$ converges to $\tilde\varphi_{KE}$  smoothly on $\tilde X \setminus (Supp E\cup Supp D)$ as $\epsilon \rightarrow 0$.

\end{lemma}

\begin{proof} The $C^0$-estimate immediately follows from \cite{EGZ, Z}. Standard calculations give
\begin{eqnarray*}
&&\Delta_\epsilon\log  tr_{\tilde \omega} (\tilde\omega_\epsilon) \\
&=& \frac{1}{tr_{\tilde\omega}(\tilde\omega_\epsilon)} \left(- (\tilde g)^{i\bar j} (\tilde R_\epsilon)_{i\bar j}+  (g_\epsilon)^{i\bar j} (\tilde g_\epsilon)_{k\bar l} \tilde R_{i\bar j}~^{k\bar l} + (\tilde g_\epsilon)^{i\bar j} (\tilde g )^{k\bar l} (g_\epsilon)^{p\bar q} \tilde\nabla_i (\tilde g_\epsilon)_{k\bar q} \tilde\nabla_{\bar j} (\tilde g_\epsilon)_{p\bar l}  - \frac{|\nabla_\epsilon tr_{\tilde \omega}(\tilde\omega_\epsilon)|^2}{tr_{\tilde\omega}(\tilde\omega_\epsilon)}  \right) \\
&\geq& -C tr_{\tilde\omega_\epsilon}(\tilde \omega) - A (tr_{\tilde\omega}(\tilde \omega_\epsilon))^{-1}%
\end{eqnarray*}
for some uniform constant $C>0$ using the upper bound for $(\tilde R_{\epsilon})_{i\bar j}$, where $(\tilde R_\epsilon)_{i\bar j}$ is the Ricci tensor of $\tilde\omega_\epsilon$, $\tilde R _{i\bar j} ~^{k\bar l}$ is the holomorphic curvature tensor of $\tilde \omega$. 
Then lemma follows by the maximum principle applied to the following quantity
$$H = \log \left(  |\sigma_{D'}|^{2B}_{h_{D'}} tr_{\tilde \omega} (\tilde\omega_\epsilon) \right) - B^2 \tilde \varphi_\epsilon$$
for some fixed sufficiently large $B>0$. The higher order estimates and local convergence follow from standard argument.

\end{proof}

Let $B_O$ be a sufficiently small Euclidean ball on $Z$ centered at $O$ such that the divisor $F$ containing $O$ can be locally defined as $z_1=0$ in $B_O$, where $z=(z_1, ..., z_n)$ are local Euclidean holomorphic coordinates on $B_O$.  Let $\tilde B_O = \pi_2^{-1}(B_O)$ in $\tilde X$. The proper transformation of $F$ is given by $$\tilde F = (\pi_2)^{-1} (F) - E. $$ Then $\tilde F$ can be locally defined as the $w =0$ for a holomorphic function $w$.

Lemma \ref{483} immediately implies the following corollary. 

\begin{corollary} \label{co43} Let $\tilde B_O= \pi_2^{-1}(B_O)$. There exist $\lambda, C>0$ such that for all $\epsilon\in (0,1)$, 
\begin{equation}
\omega_\epsilon|_{\partial \tilde B_O}  \leq  C \left(  |w|^{2\lambda} \tilde \omega \right)|_{\partial \tilde B_O}. 
\end{equation}

\end{corollary}

Let $\hat\omega$ be the smooth closed nonnegative closed $(1,1)$-form as the pullback of the Euclidean metric $\sqrt{-1} \sum_{j=1}^n dz_j \wedge d\bar z_{j}$ on $B_O$. $\hat\omega$ is K\"ahler on $\tilde B_O \setminus E$. 

\begin{lemma} \label{l239}There exist $C>0$, sufficiently small $\epsilon_0>0$ and a smooth hermitian metric $h_E$ on $E$ such that on $\tilde B_O$, 
\begin{equation}\label{locba}
C^{-1} \hat\omega \leq \tilde\omega \leq C |\sigma_E|^2_{h_E}  \hat\omega, 
\end{equation}
\begin{equation}\label{locbaa} %
\tilde \chi - \epsilon_0 Ric(h_E) > C^{-1}\tilde \omega. 
\end{equation}

\end{lemma}

\begin{proof} (\ref{locba}) is a local estimate by straightforward calculations and we refer the details to \cite{SW2}. $[\tilde\chi]$ is a  trivial cohomology class on $B_O$, and $[\tilde\chi] - \epsilon [E]$ is positive on $B_O$ for sufficiently small $\epsilon>0$ and so we can choose a smooth hermitian metric $h_E$ satisfying (\ref{locbaa}).

\end{proof}

The following proposition is the main result of this section. 

\begin{proposition} \label{digest}   There exist $\delta >0$ and $C>0$ such that for any solution $\tilde\omega_\epsilon$ of equation (\ref{4365}) with $\epsilon\in (0, 1)$, we have on $\tilde B_O$, 
\begin{equation}
\tilde \omega_\epsilon \leq \frac{C}{|\sigma_E|^{2(1-\delta)}_{h_E} |w|^{2\lambda}} \tilde\omega. 
\end{equation}

\end{proposition}

\begin{proof}

 Let $$H_\epsilon = \log |\sigma_E|_{h_E}^2 |w|^{2\lambda} tr_{\hat \omega}(\tilde \omega_\epsilon) - A \tilde \varphi_\epsilon $$ for some sufficiently large $A>0$.
 Then standard calculations show that 
 \begin{eqnarray*}
 \Delta_\epsilon H_\epsilon &\geq & - C~\frac{ tr_{\hat \omega}(\tilde \omega ) }{tr_{\hat\omega}( \tilde \omega_\epsilon)}  +  tr_{ \tilde \omega_\epsilon}(A\tilde \chi - Ric(h_E)) - An \\
 &\geq& - \frac{C}{| \sigma_E|^2_{h_E} tr_{\hat\omega} ( \tilde \omega_\epsilon)} + c~ tr_{\tilde \omega_\epsilon}(\tilde\omega) - An 
 \end{eqnarray*}
 for some uniform $c>0$.
 Let $$G_\epsilon = H_\epsilon + \delta \log \left( |w|^{2\lambda } tr_{\tilde \omega}( \tilde \omega_\epsilon) \right). $$
 For fixed sufficiently large $\lambda>0$, there exists $C>0$ such that for any $\epsilon\in (0,1)$, we have
 $$ \sup_{\partial \tilde B_O} G_\epsilon \leq C $$
from the estimates in Corollary \ref{co43}. By standard calculations, Lemma \ref{riup} and estimate (\ref{locba}), there exist uniform constants $c, C>0$ such that 
 \begin{equation}
 \Delta_\epsilon G_\epsilon \geq   - \frac{C}{| \sigma_E|^2_{h_E} tr_{\hat\omega} (\tilde \omega_\epsilon)} + c~ tr_{\tilde \omega_\epsilon}(\tilde\omega) - An. 
 \end{equation}
We can assume that  $$\sup_{\tilde B_O}G_\epsilon = G_\epsilon( p_{max})$$ for some $p_{max}\in \tilde B_O$.  Then at $p_{max}$, 
$$(tr_{\tilde \omega_\epsilon}(\tilde\omega) - A^2) |\sigma_E|^2_{h_E} tr_{\hat\omega}( \tilde \omega_\epsilon )  \leq C, $$
and apply mean value inequality and the uniform upper bound for $(\tilde\omega_\epsilon)^n$ in terms of $(\tilde\omega)^n$, we have 
$$\left( (tr_{\tilde\omega}(\tilde \omega_\epsilon) )^{\frac{1}{n-1} }- A^3 \right) | \sigma_E|^2_{h_E} tr_{\hat\omega} (\tilde \omega_\epsilon)\leq C. $$
Without loss of generality, we can assume that $$\left( tr_{\tilde\omega }(\tilde \omega_\epsilon) \right)^{1/(n-1)}  >2A^3 , $$ otherwise, $G_\epsilon$ will be bounded above by a uniform constant.  Then 
$$G_\epsilon(p_{max}) \leq C$$
and so $G_\epsilon$ is uniformly bounded above on $\tilde B_O$. The proposition immediately follows.

\end{proof}

The following corollary immediately follows from Proposition \ref{digest} by letting $\epsilon \rightarrow 0$.

\begin{corollary} \label{co44} There exist $\lambda, C>0$ such that on $\tilde B_O$, we have 
\begin{equation}
(\tilde\pi)^* \omega_{KE}  \leq \frac{C}{|\sigma_E|^{2(1-\delta)}_{h_E} |w|^{2\lambda}} \tilde\omega. 
\end{equation}

\end{corollary}

\begin{corollary} \label{co45} Suppose $diam_{d_\infty}(X_\infty) = \infty$. Then there exists a sequence of points $q_j \in (X_\infty)_{reg} $ satisfying 

\begin{enumerate}

\item $d_\chi(\Phi(q_j), P) \rightarrow 0 $ for some point $P\in X\setminus \Phi(X_\infty)$,
\medskip

\item there exists $C>0$ such that for all $j$
$$
d_\infty(p_\infty, q_j) \leq C. 
$$

\end{enumerate}

\end{corollary}

\begin{proof}  From the assumption,  $X  \setminus \Phi(X_\infty)  \neq \phi$. We can choose $O$ such that $\pi_1(O) \in X_{sing}\setminus \Phi(\mathcal{S}) $. We pick a point $q$ on $E\setminus \tilde F$ in $B_O$. Then in a local chart $U$ of $q\in B_O$, $E$ is defined by $z_1=0$, where $z_1, ..., z_n$ are local holomorphic coordinates on $U$ and $q=0$. Then 
$$(\tilde\pi)^*\omega_{KE} \leq C |z_1|^{- (2-\delta)} \sum dz_i\wedge d\bar z_i $$
for some fixed $C$ and $\delta>0$. Obviously, for any point $q_j$ converging to $q$ in $U\setminus E \subset X_{reg}$, 
$$d_\infty(q_j, p_\infty) \leq C$$
for some uniform constant $C>0$. On the other hand, $ \tilde\pi(q)=  \pi_1(O) \in X_{sing}\setminus \Phi(\mathcal{S})$, so 
$d_{\chi} (q_j, \tilde\pi(q))) \rightarrow 0$.

\end{proof}

\begin{theorem} \label{diam} $(X_\infty, d_\infty)$ is a compact metric length space. 

\end{theorem} 

\begin{proof} Suppose not. Then by Lemma \ref{46}, there exists $P\in X\setminus \Phi(X_\infty)$. We can assume there exists a log resolution $\pi_1: Z \rightarrow X$ and there exists a point $O$ with $\pi_1(O)=P$ and $O$ lying in the smooth part of the exceptional divisor of $\pi_1$. Then we can apply Corollary \ref{co45} and there exist a sequence of points $q_j\in (X_\infty)_{reg}$ such that 
$$\lim_{j\rightarrow \infty} d_\chi(\Phi(q_j), P) =0, ~ \limsup_{j\rightarrow \infty} d_\infty(q_j, p_\infty) <\infty. $$
But this contradicts Lemma \ref{46}. 

\end{proof}

\noindent {\it Proof of Theorem \ref{main2}. }  From Theorem \ref{diam}, there exists $D>0$ such that 
$$diam_{g_j}(\tilde X) \leq C .$$ From the argument in section 4.2 and 4.3, $\Phi$ is both surjective and injective. The proof is complete by applying Corollary \ref{co42}.


\subsection{Applications and generalizations} 

In this section, we will discuss the application of Theorem \ref{main2} applied to the K\"ahler-Ricci flow on smooth minimal models of general type and the generalization of Theorem \ref{main2} to canonical pairs with crepant singularities.

\noindent {\it Kahar-Ricci flow on minimal models of general.} Let $X$ be an $n$-dimensional smooth minimal model of general type. Then the pluricanonical system induces  a birational morphism
$$\pi: X \rightarrow X_{can}, $$
where $X_{can}$ is the unique canonical model $X_{can}$ of $X$. In particular, $X_{can}$ has crepant singularities as $\pi$ is a crepant resolution. Let $E$ be the exceptional locus, i.e., where $\pi$ is not isomorphism. We now consider the following normalized K\"ahler-Ricci flow for any initial K\"ahler metric $g_0$
\begin{equation}\label{krf2}
 \ddt{g} = - Ric(g) - g, ~ g(0) = g_0. 
 \end{equation}
We now can prove Corollary \ref{main6}

\begin{proposition} The K\"ahler-Ricci flow (\ref{krf2}) admits a smooth solution $g(t)$ for $t\in [0, \infty)$ and $g(t)$ converges smoothly to a K\"ahler metric $g_\infty$ on $X\setminus E$. The metric completion of $(X\setminus E, g_\infty)$ is a singular K\"ahler-Einstein metric length space homeomorphic to the canonical $X_{can}$.

\end{proposition}

\begin{proof} The long time existence and local smooth convergence of $g(t)$ is due to Tsuji \cite{Ts}. Since the limiting K\"ahler-Einstein current has bounded local potential, we can apply Theorem \ref{main2} and complete the proof of the proposition.

\end{proof}

The drawback of the above proposition is that we do not have much geometric information for the global solution $g(t)$ such as a uniform diameter bound. The following conjecture is well-known.

\begin{conjecture} $(X, g(t))$ converges to the metric complexion $(X_{can}, g_\infty)$ in Gromov-Hausdorff topology as $t\rightarrow \infty. $

\end{conjecture}

\medskip

\noindent {\it Log canonical pairs.} We would  like to remark that both Theorem \ref{main1} and Theorem \ref{main2} and be generalized to a pair $(X, D)$, where $D$ is a divisor of simple normal crossings.  In general, one can consider a normal projective variety $X$ with crepant singularities paired with an effective divisor of simple normal crossings $D=\sum_{i=1}^I a_iE_i - \sum_{j=1}^J F_j$, where $a_i>-1$, $E_i$ and $F_j$ are smooth prime divisors for all $i, j$. We assume that $K_X+D$ is ample. By \cite{B}, there exists a unique K\"ahler-Einstein current $\omega_{KE}$ on $X\setminus F$ with $\int_{X\setminus D} \omega_{KE}^n = [K_X+D]^n$, where $F=\sum_{j=1}^J F_j$. Using the argument in the proof of Theorem \ref{main2} and the argument in section 5, one should be able to show that the metric completion $(X_{reg}\setminus D, \omega_{KE})$ is homeomorphic to the quasi-projective variety $X\setminus F$.

Another interesting direction is to study the Riemannian geometry of the twisted K\"ahler-Einstein current on a canonical model of non-general type introduced by the author and Tian \cite{ST1, ST2}. We expect the gradient estimate similar to Proposition \ref{gradeke} to hold in this collapsing case as well.


\section{Degeneration of K\"ahler-Einstein manifolds of general type}

Suppose $\Psi: \mathcal{X} \rightarrow B$ is a flat projective family of $n$-dimensional canonically polarized varieties of general type over an open disc in $\mathbb{C}$. We assume that 
\begin{enumerate}

\item $X_t = \Psi^{-1}(t)$ is smooth for $t\in B^*=B\setminus \{0\}$. We let $g_t\in -c_1(X_t)$ be the unique K\"ahler-Einstein metric on $X_t$ for $t\in B^*$. 

\smallskip

\item $\mathcal{X}$ has at worst canonical singularities. $X_0$ is reduced and irreducible with at worse log canonical singularities. 

\smallskip

\item the relative canonical sheaf $K_{\mathcal{X}/B}$ is an ample $\mathbb{Q}$-line bundle with $(K_{\mathcal{X}/B} )|_{X_t} = K_{X_t}$ for all $t\in B$.

\smallskip

\item  for any $k$, $R$ and any smooth holomorphic section $\eta\in H^0(\mathcal{X}, (K_{\mathcal{X}/B})^k)$, 
$$\sup_{t\in B^*} \int_{B_{g_t}(p_t, X_t)} \left|\eta|_{X_t} \right|^{2/k} < \infty, $$
where $p_t$ is a continuous family of points in $\mathcal{X}$ with $p_0\in (X_0)_{reg}$ and $p_t\in X_t$, $B_{g_t}(p_t, X_t)$ is the geodesic ball in $X_t$ of radius $R$ centered at $p_t$. 
 When $X_0$ is log terminal, we assume that for any $k$, $R$ and any smooth holomorphic section $\eta\in H^0(\mathcal{X}, (K_{\mathcal{X}/B})^k)$,
$$\sup_{t\in B^*} \int_{X_t } \left|\eta|_{X_t} \right|^{2/k} < \infty.$$

\end{enumerate}

We remark that (4) is a technical assumption and one might be able to remove it. In particular, when $X_0$ is log terminal, the assumption might always hold by applying arguments of Gross in \cite{RZ} (Theorem B.1 (ii)). 

After embedding $\mathcal{X}$ into $\mathbb{CP}^{d_k}$ by the $k$-power of $K_{\mathcal{X}/B}$ for sufficiently large $k$, we let $\chi\in -c_1(K_{\mathcal{X}/B})$ be a smooth K\"ahler form on $\mathcal{X}$ induced from the projecting embedding. Then $\chi_t = \chi|_{X_t}$ is a smooth K\"ahler form in $-c_1(X_t)$ for $t\in B^*$. We can pick a smooth real valued $(n,n)$-form $\Omega$ on $\mathcal{X}$ such that $\Omega_t = \Omega|_{X_t}$ is a smooth non degenerate volume form on $X_t$ for $\in B^*$. $\Omega$ can be expressed by $\sum_j (\eta_j \wedge \overline \eta_j)^{1/k} $, where $\eta_j \in H^0(\mathcal{X}, (K_{\mathcal{X}/B})^k) $ and $(K_{\mathcal{X}/B})^k$ is generated by $\{ \eta_j   \}_j$ over $B$, for some sufficiently large $k$. We can then assume that on $\mathcal{X}$, 
 $$\ddbar \log \Omega = \chi. $$
 We now consider the following family of complex Monge-Ampere equations on $X_t$ for $t\in B$.
 \begin{equation} \label{famke}
 (\chi_t + \ddbar\varphi_t)^n = e^{\varphi_t} \Omega_t
 \end{equation}
Obviously equation (\ref{famke}) admits a unique smooth solution for all $t \in B^*$. We let $\omega_t = \chi_t+ \ddbar \varphi_t$ be the K\"ahler-Einstein form for $t\in B^*$, and let $g_t$ be the associated K\"ahler metric.


\subsection{A priori estimates}

The following lemma can be calculated locally through local embedding by the same argument in \cite{RZ} (Theorem B.1 (i) ). 

\begin{lemma} There exists $C>0$ such that for all $t\in B^*$, we have on $X_t$
$$(\chi_t)^n \leq C \Omega_t.$$

\end{lemma}

Then immediately by the maximum principle applied to the equation (\ref{famke}), we have a uniform upper bound for the potential $\varphi_t$. 

\begin{corollary} \label{5upb}There exists $C>0$ such that for all $t\in B^*$, we have on $X_t$
$$\varphi_t \leq C. $$

\end{corollary}

We pick a base point $p_0 \in (X_0)_{reg}$. Then there exists a sequence of points $p_t$ with $d_\chi(p_t, p_0) \rightarrow 0$ as $t\rightarrow 0$ with respect to the metric $\chi$ on $\mathcal{X}$. Let $B_t(R)= B_{g_t}(p_t, R)$ be the geodesic ball in $(X_t, g_t)$ centered at $p_t$ with radius $R$.

\begin{lemma} \label{5lob}For any $R>0$, there exists $C_R>0$ such that for all $t\in B^*$
$$\varphi_t \geq - C_R$$
on $B_t(R)$. 
\end{lemma}

\begin{proof}

We   apply the argument in section 4.3 because the Sobolev constant on $B_t(R)$ is uniformly bounded for all $t\in B^*$. Applying Moser's iteration to $e^{-\delta \varphi_t}$ independent of $t$ with respect to $\omega_t$, there exist $C_1, C_2>0$ such that 
$$|| e^{-\delta \varphi_t}||_{L^\infty(B_t(R))} \leq C_1 \int_{B_t(2R)} e^{-2\delta \varphi_t} \omega_t^n  = C_1 \int_{B_t(R)}   \Omega_t \leq C_2, $$
where the last two inequalities from the assumption (4)  and the uniform upper bound for $\varphi_t$ in the beginning of section 5. 

\end{proof}

After applying the stable reduction, we can assume there exists a log resolution $\pi: \tilde{\mathcal{X}} \rightarrow \mathcal{X}$ 
such that $\tilde \Psi: \mathcal{X}\rightarrow B$ commutes with $\pi$ and $\Psi$, both $\mathcal{X}$ and $X_0$ are smooth, and $\pi$ is normal crossing, i.e., $\tilde X_0=\tilde\Psi^{-1}(0)$ is the union of smooth divisors of simple normal crossing.  Let $\pi_0: \tilde X_0 \rightarrow X_0$ be the induced resolution of $X_0$. Then the central fibre $\tilde\Psi^{-1}(0) = \tilde X_0 \cup E $, where $E = \cup_j E_j$ is the union of prime divisors of simple normal crossings. Let $h_E$ be a smooth hermitian metric on the line bundle associated to $E$. 
Let $\tilde \chi = \pi^* \chi$. By Kodaira's lemma, there exists an effective divisor $\tilde E$ whose support coincides with the support of $E$ and a smooth hermitian metric $h_{\tilde E}$ equipped on the line bundle associated to $\tilde E$ such that $ \tilde \chi_\epsilon = \pi^* \chi - \epsilon Ric(h_{\tilde E})$ is a K\"ahler form on $\mathcal{X}$ for all sufficiently small $\epsilon >0$. Let $\sigma_{\tilde E}$ be the defining section of $\tilde E$.

\begin{lemma} \label{54} For any $\epsilon>0$, there exists $C_\epsilon>0$ such that on $X_t$
$$ \tilde \varphi_t \geq \epsilon \log |\sigma_{\tilde E} |_{h_{\tilde{E}}}^2 - C_\epsilon, $$
for all $t\in B$.

\end{lemma}  

\begin{proof} Let $\sigma_{E_j}$ be the defining section for $E_j$ and $h_{E_j}$ the smooth hermitian metric on the line bundle associated to $E_j$ such that 
$$\tilde\chi_{t, \epsilon}  = \tilde\chi_t - \epsilon \ddbar \log Ric(h_{\tilde E} )+  \epsilon^2 \delta  \ddbar \left(  |\sigma_{E_j}|^{2\delta}_{h_{E_j}} \right) . $$
is a K\"ahler metric on $\tilde{\mathcal{X}}$ with conical singularities along $E_j$ for fixed $\delta>0$ sufficiently close to $0$ and sufficiently small $\epsilon>0$. Let 
$$\tilde \varphi_{t, \epsilon} = \tilde \varphi_t - \epsilon \log |\sigma_{\tilde E}|^2_{h_{\tilde E}} -\epsilon^2 \delta \ddbar  |\sigma_{E_j}|^{2\delta_j}_{h_{E_j}} . $$ 
Then we have for $t\in B^*$, 
$$
 \pi^*( dt\wedge d\bar t) \wedge (\tilde \chi_{t, \epsilon} + \ddbar \tilde\varphi_{t, \epsilon} )^n = e^{\tilde \varphi_{t, \epsilon} + \epsilon \log |\sigma_{\tilde E}|^2_{h_{\tilde E}} + \epsilon^2 \left( |\sigma_{E_j}|^{2\delta_j}_{h_{E_j}} \right) } \pi^* (dt\wedge d\bar t\wedge  \Omega).
$$
We can now apply the maximum principle to $\tilde \varphi_{t, \epsilon}$. Suppose $\tilde \varphi_{t, \epsilon}$ achieves its minimum at $p_{t, \epsilon}$. Then 
\begin{equation}\label{lwer}  
e^{\tilde\varphi_{t, \epsilon}(p_{t, \epsilon}) } \geq \left.  e^{- \epsilon \log |\sigma_{\tilde E}|^2_{h_{\tilde E}} - \epsilon^2 \left( |\sigma_{E_j}|^{2\delta_j}_{h_{E_j}} \right)} \frac{ \pi^*(dt\wedge d\bar t) \wedge (\tilde\chi_{t, \epsilon})^n}{ \pi^* (dt \wedge d\bar t  \wedge \Omega)} \right|_{p_{t, \epsilon}}.
\end{equation}
Without loss of generality, we can assume $p_{t, \epsilon}$ lies close to the exceptional locus of $\pi$. By the normal crossing of $\pi$, we can assume that near $p_{t, \epsilon}$,
 $$\pi^{-1} (t) = x z_1z_2... z_m,  $$
where locally, $ x, z_1, ..., z_{n}$ are local holomorphic coordinates near $p'=0$, $1\leq m \leq n$ and locally   $\tilde X_0$ is defined by $x=0$ and the exceptional divisor $E'_i$ is defined by $z_i=0$, $i=1, ..., m$. Then 
$$  dt\wedge d \bar t = (z_1...z_m dx + \sum_{i=1}^m \frac{x z_1... z_m} {z_i} dz_i )\wedge (\bar z_1... \bar z_m dx + \sum_{i=1}^m \frac{\bar x \bar z_1... \bar z_m}{\bar z_i} d \bar z_i ), $$
(\ref{lwer}) locally becomes 
\begin{eqnarray*}
&&e^{\tilde\varphi_{t, \epsilon}(p_{t, \epsilon})} \\
& \geq & c_1 \frac{ ( |z_1...z_m|^2 + \sum_{i=1}^m \frac{|xz_1...z_m|^2}{|z_i|^2} ) (|z_1|^2...|z_m|^2)^{-(1-\delta)- \epsilon} dx\wedge d\bar x\wedge dz_1\wedge d\bar z_1 \wedge... \wedge  dz_{n}\wedge d\bar z_{n} }{ dx\wedge d\bar x\wedge dz_1\wedge d\bar z_1 \wedge... \wedge  dz_{n}\wedge d\bar z_{n}}\\
&\geq &c_2
\end{eqnarray*}
for some fixed $c_1, c_2>0$ by choosing $\delta>0$ sufficiently small.  The lemma follows after the above argument over finitely many local neighborhoods.

\end{proof}

\begin{lemma} \label{55} For $k>0$ and any compact set $K\subset\subset \mathcal{X} \setminus \left( (X_0)_{sing} \cup \mathcal{X}_{sing} \right)$, there exists $C_{k,K}$ such that 
$$||\varphi_t||_{C^k(K\cap X_t, \chi_t)} \leq C_{k,K}. $$

\end{lemma}
\begin{proof} This lemma can be easily proved by Tsuji's trick. The standard maximum principle applied to  $H = \log |\sigma_{\tilde E}|^{2\lambda} tr_{\tilde \chi_\epsilon} (\pi^*\omega_t) - A \pi^*\varphi_t$ on $X_t$ for sufficiently large $\lambda$ and $A>0$ shows that $\omega_t$ is uniformly bounded on $K$. The higher regularity follows from standard local estimates. 
\end{proof}

\begin{definition}\label{slc}
Suppose $K_{\tilde X_0} = (\pi_0)^* K_{X_0} + \sum_{i=1}^I a_i E_i + \sum_{j=1}^J b_j F_j$, where $E_i$ and $F_j$ are all the smooth prime divisors in the exceptional locus of $\pi_0$ with $a_i > - 1$ and $b_j=-1$ for all $i, j$. 
$$S_{lt} = (\pi_0)^{-1} \left( \cup_{i=1}^I E_i \setminus \cup_{j=1}^J F_j \right) , ~  S_{lc} = (\pi_0)^{-1} \cup_{j=1}^J F_j, ~S^\circ_{lc} = (\pi_0)^{-1} \left( \cup_{j=1}^J F_j \setminus \cup_{i=1}^I E_i \right).$$

\end{definition}
We remark that the definition of  $S_{lt}$, $S_{lc}$ does not depend on the choice of resolutions. We  eventually want to show that the potential of the K\"ahler-Einstein current is locally bounded on $S_{lt} $ and must tend to $-\infty$ near $S_{lc}$.

\begin{definition}
We define   $\mathcal{P}(X_0, \chi_0) $ to be the set of all $\varphi\in PSH(X_0, \chi_0)~\cap~ C^\infty((X_0)_{reg})$ satisfying: for any $\epsilon>0$, there exists $C_\epsilon>0$ such that on $\tilde X_0$, 
$$(\pi_0)^* \varphi \geq \epsilon \log |\sigma_{\tilde E}|^2_{h_{\tilde E}} -C_\epsilon. $$

\end{definition}

\begin{proposition} \label{5uniq} There exists a unique $\varphi_0 \in \mathcal{P}(X_0, \chi_0)$ solving 
\begin{equation}\label{eqqe}
(\chi_0 + \ddbar \varphi_0)= e^{\varphi_0} \Omega_0 
\end{equation}
on $(X_0)_{reg}$. In particular, $\varphi_0$ satisfies the following
\begin{enumerate}

\smallskip

\item $\int_{X_0} e^{\varphi_0} \Omega_0 = \int_{X_0} \chi_0^n = (-c_1(X_0))^n, $

\smallskip

\item $\varphi_0$ tends to $-\infty$ near $S_{lc}^{\circ}$.

\end{enumerate}

\end{proposition} 

\begin{proof}  The existence follows immediately from Lemma \ref{54} and Lemma \ref{55} by taking a convergent subsequence of $\varphi_t$ for $\in B^*$.  By the results of Berman and Guenancia \cite{B}, there exists a unique solution $\varphi_0'$ satisfying 
$$(\chi_0+ \ddbar \varphi_0')^n = e^{\varphi_0'} \Omega_0, ~ \int_{X_0} e^{\varphi_0'} \Omega_0 = (-c_1(X_0))^n. $$
We would like to show that $\varphi_0= \varphi_0'$
By the stability results in \cite{B} (Theorem 3.4), for any $\epsilon>0$, there exists an approximating solution $\varphi'_{0, \epsilon} \in \mathcal{P}(X_0, \chi_0)$ solving 
$$(\chi_0+ \ddbar \varphi'_{0, \epsilon})^n = e^{(1-\epsilon) \varphi_{0, \epsilon}'}  \Omega_0, ~~\int_{X_0}e^{(1-\epsilon)\varphi_{0, \epsilon}'}\Omega_0 = (-c_1(X))^n, ~\varphi'_{0, \epsilon} \rightarrow \varphi_0'.$$
The fact that $\varphi'_{0, \epsilon}\in \mathcal{P}(X_0, \chi_0)$ follows from the maximum with barrier functions and similar argument in the proof of Lemma \ref{54} after regularizing $\Omega_0$. 

We claim that $\varphi_0 \geq \varphi_0'$. Let $\psi_\epsilon = \varphi_0 - (1-\epsilon) \varphi_{0, \epsilon}' - \epsilon^2 \log |s_{\tilde E} |^2_{h_{\tilde E}}. $ Then 
$$\frac{ \left( (1-\epsilon) (\chi_0 + \ddbar \varphi'_{0, \epsilon}) + \epsilon \chi_0 - \epsilon^2 Ric(h_{\tilde E}) + \ddbar \psi_\epsilon \right)^n}{ (\chi_0 + \ddbar \varphi'_{0, \epsilon})^n} = e^{\psi_\epsilon + \epsilon^2 \log |s_{\tilde E}|^2_{h_{\tilde E}}}. $$
Since $\psi_\epsilon \in C^\infty(X_0\setminus \tilde E)$, we can apply the maximum principle and we have 
$$\psi_\epsilon \geq \log n (1-\epsilon) - \epsilon^2 \sup_{X_0} \log |s_{\tilde E}|^2_{h_{\tilde E}}. $$
Then our claim follows by letting $\epsilon \rightarrow 0$. One can similarly show that  $\varphi_0 \leq \varphi_0'$ by replacing $\epsilon$ by $-\epsilon$. 

Once $\varphi_0 = \varphi_0'$, $\int_{X_0} e^{\varphi_0}\Omega_0 = (-c_1(X))^n$ and $\varphi_0$ must tends to $-\infty$ near $S_{lc}^{\circ}$ by Berndttson's result (Lemma 2.7, \cite{B}).

\end{proof}

The uniqueness in Proposition \ref{5uniq} immediately implies the following corollary.

\begin{corollary} \label{5smcon} $g_t$ converges to $g_0$ smoothly on any compact subset of $(X_0)_{reg}$, as $t\rightarrow 0$.

\end{corollary}

\subsection{Local $L^2$-estimates and separation of points}

We choose a fixed point $p_0\in (X_0)_{reg}$ and choose $p_t\in X_t$, $t\in B^*$, converges to $p_0$ in $(\mathcal{X}, \chi)$, as $t\rightarrow 0$.

\begin{proposition} \label{5ghcon} After passing to a subsequence, $(X_t, p_t, g_t)$ converges in Gromov-Hausdorff topology to a metric lengths space $(X_\infty, d_\infty)$ satisfying

\begin{enumerate}

\item $X_\infty = \mathcal{R}\cup \mathcal{S}$, where $\mathcal{R}$ and $\mathcal{S}$ are the regular and singular part of $X_\infty$.

\smallskip

\item $\mathcal{R}$ is open and $\mathcal{S}$ is closed with Hausdorff dimension no greater than $2n-4$. 

\smallskip

\item $(X_0)_{reg}$ is an open dense set in $\mathcal{R}$.

\end{enumerate}

\end{proposition}

\begin{proof}  From the smooth convergence of $g_t$ to $g_0$ in an open neighborhood of $p_0$, there exist $\epsilon_0>0$ and  $r_0>0$ such that 
$$Vol_{g_t} (B_{g_t}(p_t, r_0) \geq \epsilon_0.$$
Then Cheeger-Colding theory immediately implies the pointed Gromov-Hausdorff convergence and (1) and (2). Also from the smooth convergence of $g_t$ to $g_0$ on $(X_0)_{reg}$, $(X_0)_{reg}$ is an open set of $X_\infty$. By the pointed Gromov-Hausdorff convergence, the Hausdorff volume measure of $(X_\infty, d_\infty)$ is equal to $(-c_1(X_0))^n$ since $[K_{X_t}]^n$ is invariant in $t$. On the other hand, from Proposition \ref{5uniq}, 
$$\int_{(X_0)_{reg}} dV_{g_0} = \int_{(X_0)_{reg}} e^{-\varphi_0} \Omega_0 = (-c_1(X_0))^n, $$
therefore $(X_0)_{reg})$ must be dense in $(X_\infty, d_\infty)$.

\end{proof}

 The following proposition can be proved by similar arguments in Proposition \ref{ll31} as local $L^2$-estimates from Tian's proposal for the partial $C^0$-estimates. 

\begin{proposition} \label{ll41}   For any $R>0$, there exists  $K_R >0$ such that if  $s\in H^0(X_t, (K_{X_t}) ^k)$ for $k\geq 1$ with $t\in B^*$, then 
\begin{equation}\label{lll51}
\|s\|_{L^{\infty, \sharp}(B_{g_t}(p_t, R))} \leq K_R \|s\|_{L^{2, \sharp}(B_{g_t}(p_t, 2R) )}
\end{equation}
\begin{equation}\label{lll52}
 \|\nabla s\|_{L^{\infty, \sharp}(B_{g_t}(p_t, R)) } \leq K_R \|s\|_{L^{2, \sharp}(B_{g_t}(p_t, 2R))} .
\end{equation}

\end{proposition}

The following $L^2$-estimate is standard for K\"ahler-Einstein manifolds.

\begin{proposition} For $k\in \mathbb{Z}^+$, any $t\in B^*$ and any smooth $(K_{X_t})^k$-valued $(0,1)$-form $\tau$ satisfying 
$\dbar \tau =0$,  
there exists an $(K_{X_t})^k$-valued section $u$ such that $\dbar u = \tau$ and $$ \int_{X_t} |u|^2_{(h_t)^k} ~dV_{g_t} \leq \frac{1}{2\pi} \int_{X_t} |\tau|^2_{(h_t)^k}~ dV_{g_t}, $$
where $g_t \in -c_1(X_t)$ is the K\"ahler-Einstein metric on $X_t$ and $h_t$ is the hermitian metric on $K_{X_t}$ with $Ric(h_t) = g_t$.

\end{proposition}

\begin{lemma}

Let $\{ \sigma_{ j}^{(k)} \}_{j=0}^{d_k}$ be a basis of $H^0(\mathcal{X}, (K_{\mathcal{X}/B})^k)$. Then for any $R>0$, there exists $C_R>0$ such that for all $j=0, ..., d_k$, 
\begin{equation}
\int_{B_t( R)} \left| \left.\sigma_j^{(k)}\right|_{X_t} \right|^2_{h_t^k} dV_{g_t} \leq C_R.
\end{equation}

\end{lemma}

\begin{proof} The K\"ahler-Einstein hermitian metric is given by $h_t = (e^{\varphi_t} \Omega_t)^{-1}$. There exists $C_1>0$ such that $$ \left( \sum_{j=0}^{d_k} |\sigma_j^{(k)}|^2 \right)^{1/k} \leq C \Omega_t$$
since $\Omega_t$ is induced by a basis of holomorphic sections of certain power of $K_{\mathcal{X}/B}$. Also there exists $C_2>0$ such that 
$$\sup_{B_{g_t}(p_t, R)} |\varphi_t| \leq C_2 $$
by Corollary \ref{5upb} and Lemma \ref{5lob}. Therefore
$$\int_{B_{g_t}(p_t, R)} \left| \left.\sigma_j^{(k)}\right|_{X_t} \right|^2_{h_t^k} dV_{g_t} \leq (C_1)^k e^{kC_2} \int_{X_t} dV_{g_t} \leq (C_1)^k e^{kC_2} (-c_1(X_t))^n. $$

\end{proof}

\begin{lemma} Let $\{ \sigma_{ j}^{(k)} \}_{j=0}^{d_k}$ be a basis of $H^0(\mathcal{X}, (K_{\mathcal{X}/B})^k)$. Then 
$\{\sigma_{ j}^{(k)}|_{X_t} \}_{j=0}^{d_k}$ converges to linearly independent holomorphic sections $\{ \sigma_{ 0, j}^{(k)} \}_{j=0}^{d_k}$ in $H^0(X_\infty, (K_{X_\infty})^k)$. In particular, the convergence is smooth in $\mathcal{R}$ and $\sigma_{0,j}^{(k)}$ extends continuously from $\mathcal{R}$ to $X_\infty$.

\end{lemma}
\begin{proof} The convergence follows from Corollary \ref{5smcon}, Proposition \ref{5ghcon} and Proposition \ref{ll41} as well as the smooth convergence of $g_t$ on $(X_0)_{reg}$. The continuous extension follows from Proposition \ref{ll41}.

\end{proof}

\begin{lemma} For any $R>0$ and $k\in \mathcal{F}(X_0, K_{X_0})$,  there exists $\epsilon_R>0$   such that 
$$ \inf_{B_{d_\infty}(p_0, R)} \sum_{j=1}^{d_k+1} |\sigma_{0,j}^{(k)
}|^2_{h_\infty^k}\geq \epsilon_R. $$

\end{lemma}

\begin{proof} Since $k\in \mathcal{F}(X_0, K_{X_0})$, there exists $C>0$ such that on $X_{reg}$, 
$$ \left( \sum_{j=1}^{d_k+1} |\sigma_{0,j}^{(k)
}|^2 \right)^{1/k} \geq c \Omega_0.$$  Also $\varphi_0$ is uniformly bounded on $B_{d_\infty}(p_0, R)$. Therefore the lemma follows because $(X_0)_{reg}$ is dense in $X_\infty$.
\end{proof}

We can now define the map
$$\Phi: X_\infty \rightarrow X_0 \subset \mathbb{CP}^{d_k}$$
using the global basis $\{ \sigma_j^{(k, 0)}\}_{j=1}^{d_k+1} $ on $X_\infty$. Since $K_{X_0}$ is semi-ample, the map $\Phi$ is stabilized on $(X_0)_{reg}$ for sufficiently large $k \in \mathcal{F}(X_0, K_{X_0})$ and so, $\Phi|_{(X_0)_{reg}} = id. $ Therefore, $\Phi$ is the unique continuous extension of the identity map on $(X_0)_{reg}$ to $(X_\infty, d_\infty)$ and  immediately we have the following corollary from Proposition \ref{ll41}. 

\begin{corollary} $\Phi$ is a Lipschtiz map from $(X_\infty, d_\infty) $ to $(X_0, \chi_0)$.

\end{corollary}

The following proposition is the main result of this section. 

\begin{proposition} $\Phi$ is injective. 

\end{proposition}

\begin{proof} The proof proceeds by exactly the same argument of Donaldson-Sun \cite{DS} by applying the stability of the $H$-condition. The only difference is that we have to work locally on enlarging balls $B_t(R)$ and let $R\rightarrow \infty$. We do not need to raise a uniform bound on the power of $K_{X_\infty}$ because $\Phi$ is stabilized and is the unique extension of the identity map on $(X_0)_{reg}$.

\end{proof}


\subsection{Proof of Theorem \ref{main6}}

\begin{lemma} \label{l55} Suppose $diam_{d_\infty}(X_\infty) = \infty$. Then $X_0\setminus \Phi(X_\infty)$ is not empty and  for any sequence of points $q_j \in X_\infty$   with $\lim_{j\rightarrow \infty} d_\chi(\Phi(q_j), P) = 0$  for some point   $P\in X_0 \setminus \Phi(X_\infty)$,
$$d_\infty(p_\infty, q_j) \rightarrow \infty$$
as $j \rightarrow \infty$.

\end{lemma}

\begin{proof} One can apply the same argument in the proof of Lemma \ref{46}. Notice that $g_0$ is bounded below by $\chi_0$ on any fixed geodesic ball of $(X_\infty, d_\infty)$ because of the bound on $\varphi_0$ in Lemma \ref{5lob}.

\end{proof}

\begin{lemma} \label{5.09} $S_{lc}^{\circ} \cap \Phi(X_\infty) = \phi$ 

\end{lemma}

\begin{proof} We prove by contradiction. Suppose $Q\in S_{lc}^\circ ~\cap ~ \Phi(X_\infty)$. Then there exists $q\in X_\infty$ with $\Phi(q)=Q$ and there exist  a sequence $q_j\in (X_0)_{reg}$ such that $d_\infty(q_j, q) \rightarrow 0$ because $(X_0)_{reg}$ is dense in $X_\infty$. Therefore there exists $R>0$ such that $q_j \in B_{d_\infty}(p_0, R)$ and so there exists $C>0$ such that for all $j$, $$|\varphi_0(q_j)|\leq C. $$ This leads to contradiction since $\varphi_0$ tends to $-\infty$ near $S_{lc}^{\circ}$ by Proposition \ref{5uniq}.

\end{proof}

\begin{proposition} \label{co55} For any point $Q\in S_{lt}$, there exists a sequence of points $q_j \in (X_\infty)_{reg} $ satisfying 

\begin{enumerate}

\item $d_\chi(\Phi(q_j), Q) \rightarrow 0$,
\medskip

\item there exists $C>0$ such that for all $j$
$$
d_\infty(p_0, q_j) \leq C. 
$$

\end{enumerate}

\end{proposition}

\begin{proof} The proof is slightly more complicated than that of Corollary \ref{co44} and estimates in section 4.5 because $X_0$ might have singularities worse than canonical singularities, and so one cannot approximate the K\"ahler-Einstein current by smooth K\"ahler metrics with Ricci curvature uniformly bounded below as in Lemma \ref{riup}. Our strategy is to keep the pole singularities along $S_{lt}$ and obtain the approximating K\"ahler metrics with conical singularities.

For any point $q\in S_{lt}$, we may assume $q$ in the regular part of the exceptional divisor in a smooth model $Z$ of $X_0$ after a log resolution. Then we blow up $q$ by $\pi_Z: X' \rightarrow Z$ and let $\pi': X' \rightarrow X_0$. Let $E=\pi_Z^{-1}(q)$. We consider the following equation on $ X'$.  Let $D$ be the proper transform of the divisor of where $q$ lies in $Z$ and $s_D$, $s_E$ be the defining section of $D$ and $E$. We can assume the discrepancy of $D$ is $\alpha \in (-1, 0)$.  There exist smooth hermitian metric $h_E$ and $h_D$ on the line bundles associated to $[E]$ and $[D]$ such that 
\begin{equation}\label{conke}
(\chi' + \epsilon\omega' +   \ddbar \varphi_{\epsilon} )^n = e^{\varphi_\epsilon} (|s_E|^{2(n-1)}_{h_E} + \epsilon) |s_E|_{h_E}^{-2(n-1)} \Omega'_0 ,
\end{equation}
where $\chi' =( \pi')^*\chi_0$, $\Omega'= (\pi')^*\Omega_0$ and $\omega'$ is a smooth K\"ahler metric on $X'$. By standard regularization, one can show that $\varphi_\epsilon$ is smooth on $(\pi')^{-1}((X_0)_{reg})$ and for any $\delta>0$, there exists $C_\delta>0$ such that 
\begin{equation}\label{pcondi}
\varphi_\epsilon \geq \delta \log |s_{exc}|^2_{h_{exc}}- C_\delta, 
\end{equation}
where $s_{exc}$ is a holomorphic section vanishing along all the exceptional locus of $\pi'$ and $h_{exc}$ is a smooth hermitian metric. We pick a smooth point $q' \in E\setminus D$. Then locally near $q'$, $\omega_\epsilon = \chi' + \epsilon \omega' + \ddbar \varphi_\epsilon$ is equivalent to a standard conical K\"ahler metric with conical singularities of angle $2\pi (1-\alpha)$ along $D$ by applying the argument for the second order estimate in \cite{GP, DaSo}. Also we note that 
$$Ric(\omega_\epsilon) \geq - \omega_\epsilon + A\omega'$$
for some uniform constant $A>0$ for all $\epsilon>0$. The rest of the proof follows by the same argument in section 4.5. The only difference is that $\varphi_\epsilon$ satisfies (\ref{pcondi}) instead of being uniformly bounded, however, it does not affect the argument by adding barrier functions with arbitrarily small log poles.

\end{proof}

Now we can prove Theorem \ref{main6}. 

\begin{theorem} $(X_\infty, d_\infty)$ is either a metric length space homeomorphic to the quasi-projective variety $(X_0\setminus S_{lc}, \chi_0)$ if $S_{lc}\neq \phi$ or a compact metric length space homeomorphic to the projective variety $(X_0, \chi_0)$ if $S_{lc}= \phi$. In particular, 
$$(X_0)_{reg} = (X_\infty)_{reg}.$$
\end{theorem}

\begin{proof} If $S_{lc}= \phi$, then by the same argument in section 4.5  with Lemma \ref{l55} and Proposition \ref{co55}, $\Phi(X_\infty)= X_0$ and so $\Phi$ is bijection and hence homeomorphism. If $S_{lc} \neq \phi$, then by Lemma \ref{l55} and Proposition \ref{co55}, 
$$S_{lt}\subset \Phi(X_\infty).$$ 
It now suffices to show that $$( S_{lc} \setminus S_{lc}^\circ )  \cap \Phi(X_\infty)=\phi $$ by Lemma \ref{5.09}. But this must be true since $\Phi(X_\infty)$ is open in $(X_0, \chi_0)$. 

\end{proof}

Theorem \ref{main6} immediately implies the following corollary. 

\begin{corollary} $\varphi_0 \in L_{loc}^\infty(X_0\setminus S_{lc})$ and $\varphi_0$ tends to $-\infty$ along $S_{lc}$. Furthermore, if $X_0$ is log terminal, then there exists $C>0$ such that for all $t\in B$, 
$$||\varphi_t||_{L^\infty(X_t)} \leq C . $$  

\end{corollary}


\subsection{Generalizations and conjectures} The moduli space of canonically polarized varieties is a central problem in both algebraic geometry and complex geometry.  Recently, the compactification of semi-log canonical models is derived (c.f. \cite{K1, HMX}). Theorem \ref{main6} deals with a special degeneration with the central fibre being a reduced and irreducible canonical model with log canonical singularities. The proof of Theorem \ref{main6} should be generalized to much general cases when the central fibre has multiple fibres with semi-log canonical singularities. Using algebraic structures of such a compactification, one would like to prove the following statement. 

\begin{conjecture} If the central fibre of the flat degeneration of smooth canonical models is semi-log canonical, then each component of the central fibre is a pointed Gromov-Hausdorff limit of near by K\"ahler-Einstein manifolds. In particular, the pointed Gromov-Hausdorff limit is homeomorphic to the quasi-projective variety induced by a component of the central fibre. 

\end{conjecture}

The existence and uniqueness of the K\"ahler-Einstein current is derived by Berman and Guenancia \cite{B}. We hope our approach can be applied to prove the above conjecture and we will describe the problem and technical details in the sequel paper. In particular, one should be able to  establish the relation that boundedness of the local potential is equivalent to the bounded of the distance from a fixed base point. The confirmation of the above conjecture should establish the general principle that 
$$the~ algebraic ~moduli ~space ~is ~ equivalent ~ to~ the ~K\textnormal{\"a}hler\textnormal{-}Einstein ~moduli ~space$$
and so are their algebraic and geometric compactifications.  A more interesting and deeper approach would be the following differential geometric conjecture. 

\begin{conjecture} Let $\mathcal{C}(n, V)$ be the set of all canonically polarized $n$-dimensional manifolds with $c_1^n \leq V$. Then for any K\"ahler-Einstein sequence $(X_j, g_j)$ in $\mathcal{C}(n, V)$ with $Ric(g_j) = -g_j$, after passing to a subsequence, they converge to a singular K\"ahler-Einstein metric length space homeomorphic to either a projective  or quasi-projective variety in pointed Gromov-Hausdorff topology.

\end{conjecture}

\bigskip

\noindent {\bf{Acknowledgements:}} The author would like to thank Jacob Sturm for teaching him  the construction of the cut-off functions in Lemma \ref{app}. He would also like to thank Xiaowei Wang for many stimulating conversations,  D.H. Phong, Valentino Tosatti, Zhenlei Zhang, Ved Datar and Bin Guo for a number of helpful discussions.

\end{document}